\documentclass[a4paper,10pt]{article}
\usepackage[utf8]{inputenc}
\usepackage{a4wide}

\usepackage{subfigure}

\usepackage{amsmath,amssymb,graphics}
\usepackage{stmaryrd}
\graphicspath{{figures/}}

\usepackage{enumitem}
\newlist{myEnumItem}{enumerate}{4}
\setlist[myEnumItem,1]{label=\arabic*}
\setlist[myEnumItem,2]{label=\arabic{myEnumItemi}.\arabic*}
\setlist[myEnumItem,3]{label=\arabic{myEnumItemi}.\arabic{myEnumItemii}.\arabic*}
\setlist[myEnumItem,4]{label=\arabic{myEnumItemi}.\arabic{myEnumItemii}.\arabic{myEnumItemiii}.\arabic*}

\newcommand\vref{v^{\text{ref}}}

\newcommand\rhomax{\rho^{\text{max}}}

\newcommand\de{\textit{D}}
\newcommand\su{\textit{S}}

\newcommand\kmh{\frac{\text{km}}{\text{h}}}
\newcommand\ckm{\frac{\text{cars}}{\text{km}}}
\newcommand\ch{\frac{\text{cars}}{\text{h}}}

\renewcommand\P{P}
\newcommand\p{z}

\newcommand\Qt{{\tilde{Q}}}
\newcommand\Pt{{\tilde{P}}}
\newcommand\St{{\tilde{\Sigma}}}
\newcommand\qat{{\tilde{q_1}}}
\newcommand\qbt{{\tilde{q_2}}}

\newcommand\Qtt{\tilde{\tilde{Q}}}
\newcommand\Ptt{\tilde{\tilde{P}}}
\newcommand\Stt{\tilde{\tilde{\Sigma}}}
\newcommand\qatt{\tilde{\tilde{q_1}}}
\newcommand\qbtt{\tilde{\tilde{q_2}}}

\newcommand\Qla{{Q_\lambda}}
\newcommand\Pla{{P_\lambda}}

\newcommand\Qmu{{Q_\mu}}

\newcommand\las{{\lambda^*}}
\newcommand\Plas{{P_\las}}

\usepackage{graphicx}
\usepackage[breaklinks,%
					linktocpage,%
					colorlinks,%
					backref=page]{hyperref}
\hypersetup{colorlinks=true,citecolor=magenta,urlcolor=blue}
\usepackage{stmaryrd}

\newtheorem{thm}{Theorem}[section]

\newtheorem{rem}[thm]{\bf Remark}
\newtheorem{pro}[thm]{\bf Proposition}

\newtheorem{definition}{Definition}

\newenvironment{Proofc}[1]{\smallskip\par\noindent\textbf{#1}\quad}%
  {\hfill$\Box$\bigskip\par}
\newenvironment{proof}{\begin{Proofc}{Proof}}{\end{Proofc}}

\newcommand{\R}{\mathbb{R}}

\title{{P}areto-optimal coupling conditions for \\ the {A}w-{R}ascle-{Z}hang traffic flow model at junctions} 

\author{ 
Oliver Kolb\footnotemark[7], \; 
Guillaume Costeseque\footnotemark[1], \;
Paola Goatin\footnotemark[1], \; 
Simone G\"ottlich\footnotemark[7]}
\footnotetext[1]{Inria Sophia Antipolis - M{\'e}diterran{\'e}e, Universit\'{e} C\^{o}te d'Azur, Inria, CNRS, LJAD,
06902 Sophia-Antipolis, France 
(\{paola.goatin, guillaume.costeseque\}@inria.fr).}
\footnotetext[7]{University of Mannheim, Department of Mathematics, 68131 Mannheim, Germany 
(\{goettlich, kolb\}@uni-mannheim.de).}

\begin{document}

\maketitle

\begin{abstract}
This article deals with macroscopic traffic flow models on a road network.
More precisely, we consider coupling conditions at junctions for the Aw-Rascle-Zhang second order model 
consisting of a hyperbolic system of two conservation laws.
These coupling conditions conserve both the number of vehicles and the mixing of Lagrangian attributes of traffic through the junction.
The proposed Riemann solver is based on assignment coefficients, multi-objective optimization of fluxes and priority parameters.
We prove that this Riemann solver is well-posed in the case of special junctions, 
including 1-to-2 diverge and 2-to-1 merge.
\end{abstract}
%The main novelty is that the solution is required to belong to a Pareto front.


% Traffic problems, Hyperbolic equations and systems  

%--------------------------------------------------------------------------------------------------
\section{Introduction}

\subsection{Motivation}

The mathematical modeling of road traffic flows on networks has attracted an impressive interest in the community of
nonlinear Partial Differential Equations (PDE) over the last decades.
Ranging from Hamilton-Jacobi equations to hyperbolic systems of conservation laws, this field appears to be really mature.
A good review of the most recent contributions can be found in~\cite{bressan2014flows}.
The interested reader is also referred to the book~\cite{garavello2016models} and references therein.

In this work, we are interested in proposing a Riemann solver for a second order traffic flow model at junctions
and we aim at studying its well-posedness properties.
We make the usual distinction between \emph{first order models} such as 
the seminal Lighthill-Whitham-Richards (LWR) model~\cite{LighthillWhitham,richards1956shock}
that consists of a scalar conservation law for the traffic density $\rho(t,x)$, 
$t \geq 0$ and $x \in \R$ standing for the time and space variables,
\begin{equation}
\label{eq:LWR}
\partial_t \rho + \partial_x \left( \rho V(\rho) \right) = 0,
\end{equation}
and \emph{second order models} that read as a hyperbolic system of conservation laws.
Here we consider the Aw-Rascle-Zhang (ARZ) model~\cite{AwRascle2000,zhang2002non}
which combines a first conservation law for the density and another one for the mean traffic speed 
(see~\eqref{eq:ARZ} below).
In the LWR model~\eqref{eq:LWR}, the traffic speed is supposed to be always in equilibrium, i.e.\ $v = V(\rho)$, depending only on the traffic density.
The ARZ model has the advantage over the LWR model to account for observed traffic phenomena
that are due to transient traffic states
such as the capacity drop downstream a merge.
The ARZ model has already been studied in a previous article~\cite{kolb2016capacity}, where
traffic control strategies including ramp metering and variable speed limits were implemented.

To set an appropriate Riemann solver at the junction, we need to determine \emph{coupling conditions}
between incoming and outgoing roads in order to ensure the conservation of mass and momentum flows.

\newpage

\subsection{Setting}

\subsubsection{Junction}
Due to finite wave propagation speed, it is not restrictive to study a single junction.
We define a \emph{junction} $J$ as the set of $n$ incoming and $m$ outgoing branches
that meet at a single point (namely the \emph{junction point} supposed to be located at $x = 0$)
such that $J = \displaystyle \bigcup_{i = 1}^{n+m} J_i \cdot e_{i}$
where $e_i \in \R^{n+m}$ are unit vectors 
and the branch $J_{i}$ for any $i \in \{1,\ldots,n+m\}$ 
is defined as follows:
$$ J_{i} := \begin{cases}
]-\infty, 0[, \quad &\mbox{for any} \quad i = 1, \ldots, n, \\
]0,+\infty[, \quad &\mbox{for any} \quad i = n+1, \ldots, n+m.
\end{cases} $$

In the remaining, we will mainly focus on the cases of a 1-to-1 junction ($n=m=1$),
a merge ($n=2$ and $m=1$) and a diverge ($n=1$ and $m = 2$).
%As our motivation comes from road traffic modeling, we will assume that each branch $J_{i}$ has a finite length.

\subsubsection{Road dynamics}
We briefly recall the Aw-Rascle-Zhang equations~\cite{AwRascle2000,zhang2002non} 
and explain how they can be applied in the context of networks.
The model consists of a $2\times2$ system of conservation laws for the density and the velocity.
Notice that we do not include the relaxation term as proposed in~\cite{Greenberg2001}. However, it is noteworthy that adding a source term into our model will not influence the definition of solutions to Riemann problems at junctions.

%Consider a time horizon $T >0$ fixed once for all.
For the description of the traffic, we introduce the density $\rho_i=\rho_i(t,x)$
and the speed of vehicles $v_i=v_i(t,x)$ on each road ${i}$
at position $x \in J_{i}$ and time $t >0$. 
We also define the flow on road ${i}$ as
$$ q_i := \rho_i v_i,  \qquad i = 1, \ldots, n+m. $$

Then, the ARZ model on the junction $J$ reads as follows
\begin{equation} \label{eq:ARZ}
\begin{cases}
\partial_t \rho_{i} + \partial_x \left( \rho_{i} v_{i} \right) = 0, \\
\partial_t \left( \rho_{i} w_{i} \right) + \partial_x \left( \rho_{i} v_{i}  w_{i} \right) = 0, \\
w_i := v_i + p_i (\rho_i),
\end{cases} \quad \mbox{on} ~ ]0,+\infty[\, \times J_{i}, \quad i = 1, \ldots, n + m, 
\end{equation}
where $p_i(\rho)$ is a known pressure function satisfying $p_i'(\rho) > 0$ and
$\rho p''_i(\rho) + 2p_i'(\rho)>0$ for all $\rho.$
The first condition guarantees that $p_i : \rho \mapsto p_i(\rho)$ is invertible
while the latter condition ensures that the curve $\{w_i(\rho,v) = v + p_i(\rho) = c\}$ for any constant $c > 0$ 
is strictly concave in the $(\rho,\rho v)$-plane. Therefore, there exists a unique sonic point $\sigma_i(c)$ maximizing the flux $\rho v$ along the curve $\{w_i(\rho,v) = c\}$.
Notice that we also require that $p_i(0) = 0$.

We consider a Cauchy problem by supplementing~\eqref{eq:ARZ} with the initial conditions
\begin{equation} \label{eq:ARZ-Cauchy}
\left( \rho_{i} , v_{i} \right) (0, x) = g_i(x) , 
\quad \mbox{for} \quad x \in J_{i}, \quad i = 1, \ldots, n + m,
\end{equation}
for some functions $g_i : J_{i} \to \R^2$, for any $i \in \{ 1, \ldots, n + m \}$.

In conservative form, setting $y_i := \rho_i w_i$, \eqref{eq:ARZ} boils down to
\begin{equation} \label{eq:conslaw}
\begin{cases}
\partial_t \rho_{i} + \partial_x \left( y_i-\rho_ip_i(\rho_i) \right) = 0, \\
\partial_t y_i + \partial_x \left( \big( y_i - \rho_i p_i(\rho_i) \big) \dfrac{y_i}{\rho_i} \right) = 0, 
\end{cases}
   \quad \mbox{on}~ ]0,+\infty[\, \times J_{i}, \quad i = 1, \ldots, n + m.    
\end{equation}

We briefly recall the eigen-structure of the classical ARZ model (on $J_i$):
\begin{itemize}
\item the eigenvalues are $\lambda_{i,1} = v_{i} - \rho_{i} p'_{i}(\rho_{i})$ and $\lambda_{i,2} = v_{i}$ 
(with $\lambda_{i,1} < \lambda_{i,2}$ whenever $\rho_{i} > 0$),
\item the Riemann invariants are given by $W_{i,1} = v_{i}$ and $W_{i,2} = w_i = v_{i} + p_{i}(\rho_{i})$ such that
$$ \begin{cases}
\partial_t W_{i,1} + \lambda_{i,1} \partial_x W_{i,1} = 0, \\
\partial_t W_{i,2} + \lambda_{i,2} \partial_x W_{i,2} = 0.
\end{cases} $$
\end{itemize}

\begin{rem}[Lagrangian attributes]
\label{rem:lagrangian-attributes}
It is noteworthy that, for smooth solutions, the second equation of the ARZ model~\eqref{eq:ARZ} can be rewritten as 
\[
 \partial_t w_i + v_i \partial_x w_i = 0, 
\]
showing that the Lagrangian attribute $w_i$ is advected with the traffic flow.
This remark is fundamental for deriving coupling conditions on the conservation of the $\left( w_i \right)_{1 \leq i \leq n}$
through the junction point.
\end{rem}

In this paper, we will consider 
%the preferential velocity depending on the density (based on the LWR model~\eqref{eq:LWR})
%%
%\begin{equation}
% V_i(\rho) = v_i^{\max} \left( 1 - \frac{\rho}{\rho_i^{\max}} \right)
%\end{equation}
%%
%and 
the pressure function
\begin{equation}\label{eq:pressure}
 p_i(\rho) =  \frac{\vref_i}{\gamma_i} \left(\frac{\rho}{\rho_i^{\max}}\right)^{\gamma_i}  %v_i^{\max} - V_i(\rho) = \frac{v_i^{\max}}{\rho_i^{\max}} \rho .
\end{equation}
with maximal density $\rho_i^{\max}>0$, 
%maximal velocity $v_i^{\max}>0$, % only for the relaxation term
reference velocity $\vref_i>0$ and $\gamma_i>0$  (see~\cite{HautBastin2007,ParzaniBuisson2012}). 

Note that above and in the following, we use $w_i = w_i(t,x)$ as space and time dependent state variable but also as function, e.g.\ in the form $w_i(\rho,v) = v + p_i(\rho)$. 
Analogously, we will use the notation $v_i(U) = w - p_i(\rho)$ for the velocity of a state $U = (\rho,\rho w)$.

Similar to first order traffic models, we define the demand and supply functions for each road $i$ as follows: 
for a given constant $c$ (corresponding to a fixed value of $w$) we have 
\begin{align}
 \de_i(\rho,c) &= \begin{cases}
                \big( c - p_i(\rho) \big) \rho & \text{if} \ \rho \le \sigma_i(c),\\
                \big( c - p_i(\sigma_i(c)) \big) \, \sigma_i(c) & \text{if} \ \rho \ge \sigma_i(c),
               \end{cases}\label{eq:demand}\\
 \su_i(\rho,c) &= \begin{cases}
                \big( c - p_i(\sigma_i(c)) \big) \, \sigma_i(c) & \text{if} \ \rho \le \sigma_i(c),\\
                \big( c - p_i(\rho) \big) \rho & \text{if} \ \rho \ge \sigma_i(c),
               \end{cases}\label{eq:supply}
\end{align}
where
\begin{equation}
\label{eq:sonic_point}
 \sigma_i(c) = \rho_i^{\max} \left(\frac{c\, \gamma_i}{\vref_i \, (1+\gamma_i)} \right)^{\frac{1}{\gamma_i}}. %\frac{c}{2} \frac{\rho_i^{\max}}{v_i^{\max}}
\end{equation}
%
%is the sonic point on the curve $\{w_i(\rho,v) = v + p_i(\rho) \stackrel{!}{=} c\}$ in the $(\rho,\rho v)$-plane. 
%This sonic point is the density for which the flux reaches its maximum.
An illustration of the considered demand and supply functions is given in Figure~\ref{fig:demandSupply}.
Supply and demand functions are needed to compute the coupling fluxes at the junction point.

%Here and also later, we make use of the notation $v = v(U)$ and $\rho = \rho(U)$ for any given state $U = (\rho,y)$.

%Note that in the case of $\vref_i = \vmax_i$ (variant 1. above) demand and supply additionally depend on the time $t$ (which enters $p_i$ and $\sigma_i$).

\begin{figure}[hbt]
\begin{center}
\includegraphics[width=0.45\textwidth]{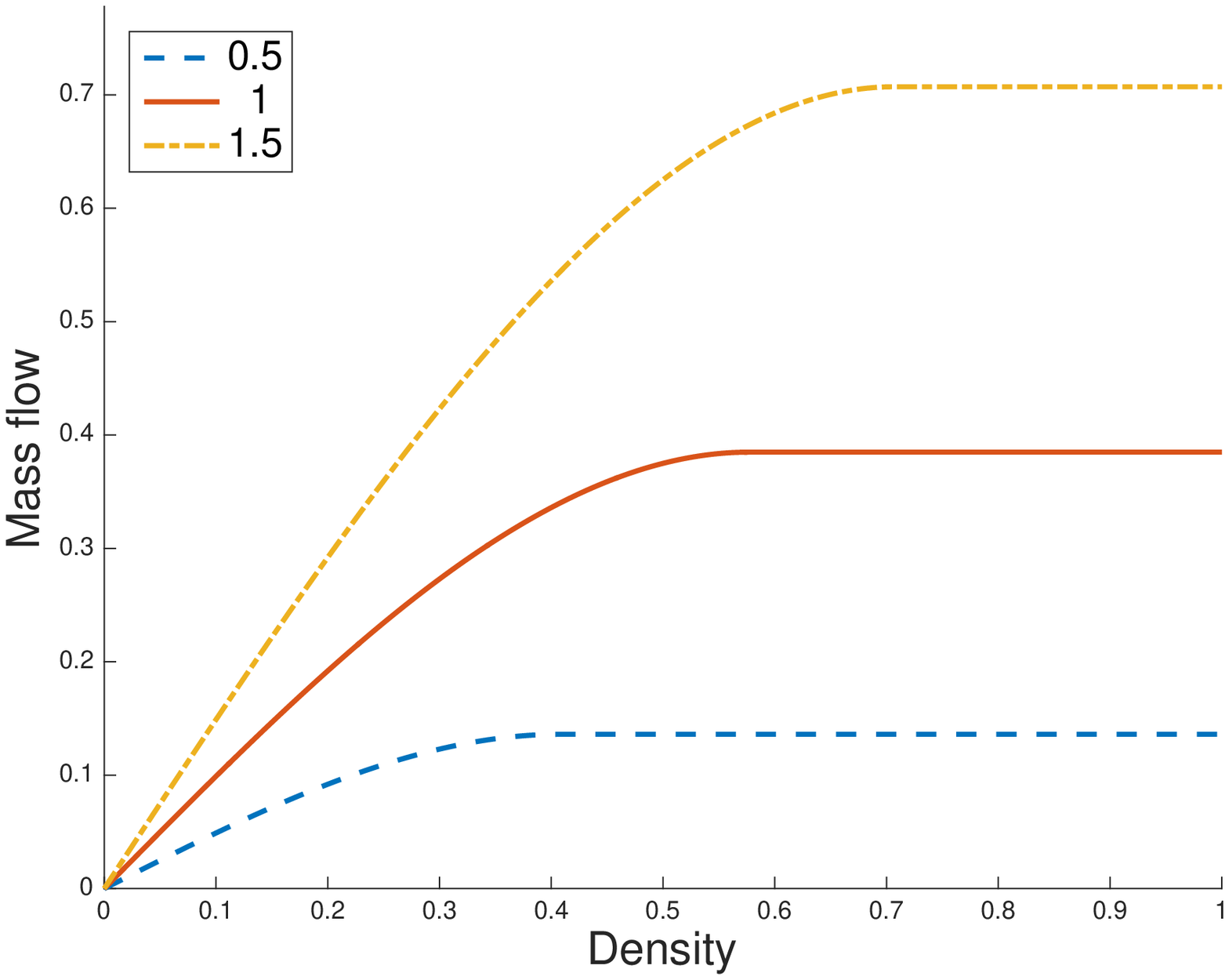}
\includegraphics[width=0.45\textwidth]{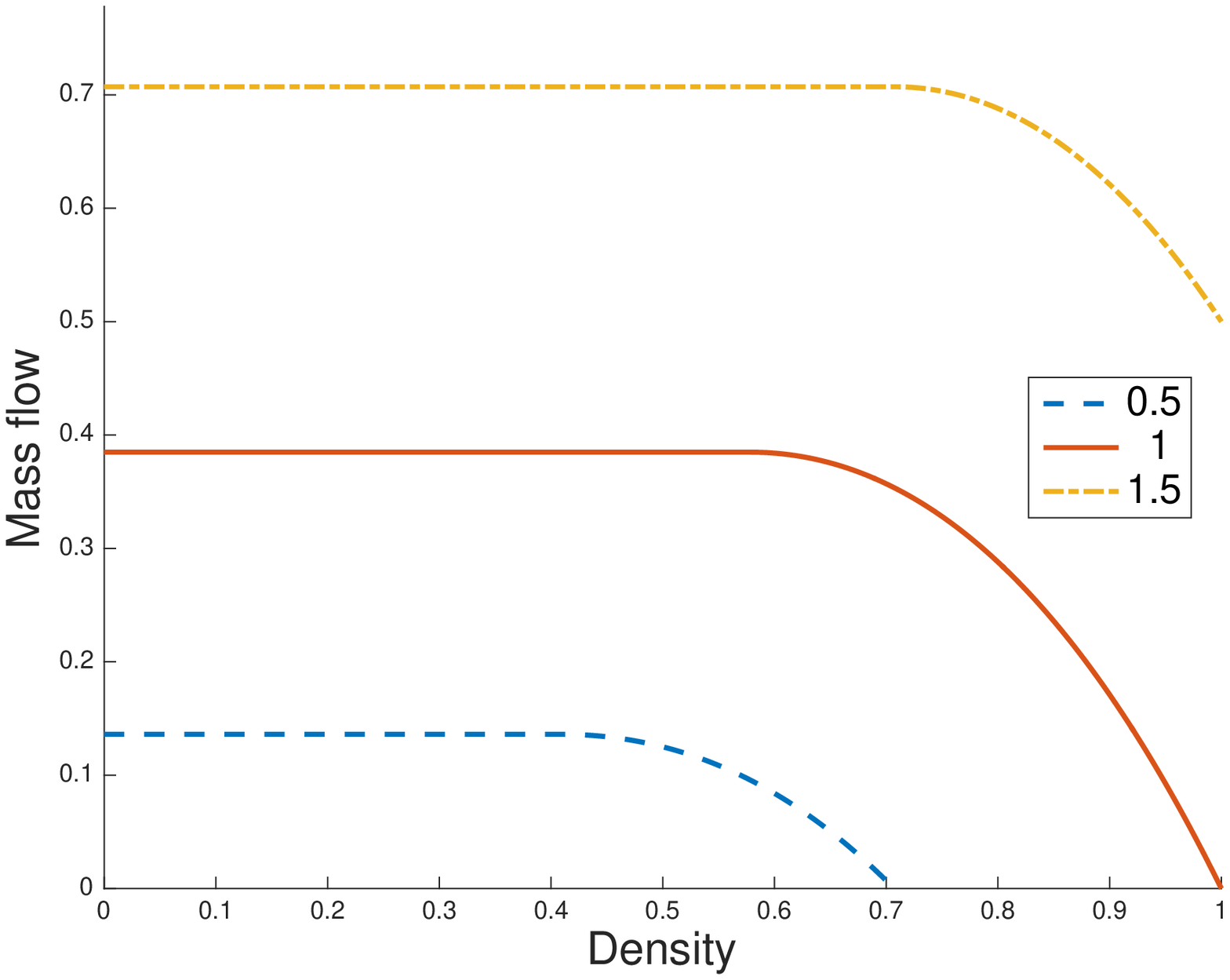}
\caption{Demand (left) and supply functions (right) on a fixed road for $\rhomax=1$, $\vref=2$, $\gamma=2$ and the given values for $c$.}
\label{fig:demandSupply}
\end{center}
\end{figure}

\subsubsection{Problem statement}
Given a junction $J$ with $n$ incoming and $m$ outgoing roads,
we look for a Riemann solver denoted by $\mathcal{RS}$ to solve a Riemann problem posed on this junction.
A Riemann problem for a junction is a generalization of a Riemann problem on a single link, 
i.e., a Cauchy problem~\eqref{eq:ARZ-Cauchy} with constant data on each branch, $g_i (x) = \left( \rho_{i,0} , v_{i,0} \right)$ for any $x \in J_{i}$ and $i \in \{ 1, \ldots, n + m \}$.
More precisely, we assume that
\begin{equation} \label{eq:ARZ-CI}
\left( \rho_{i} , v_{i} \right) (0, x) = \left( \rho_{i,0} , v_{i,0} \right) , 
\quad \mbox{for} \quad x \in J_{i}, \quad i = 1, \ldots, n + m,
\end{equation}
where $\rho_{i,0}$, $v_{i,0} \geq 0$ for any $i \in \{1,\ldots,n+m\}$.

According to the classical approach \cite{GaravelloBook2006}, we define a Riemann solver as follows:

\begin{definition}[Riemann solver at junction]
\label{def:RS}
A Riemann solver $\mathcal{RS}$ of the Riemann problem~\eqref{eq:ARZ}, \eqref{eq:ARZ-CI} on the junction $J$
is a function
$$ \begin{array}{rccc}
\mathcal{RS} : & (\R \times \R)^{n+m} & \longrightarrow & (\R \times \R)^{n+m} \\
 & \left( (\rho_{1,0}, v_{1,0}) , \ldots, (\rho_{n+m,0}, v_{n+m,0}) \right) & \longmapsto & 
 \left( (\hat{\rho}_{1} , \hat{v}_{1}), \ldots, (\hat{\rho}_{n+m} , \hat{v}_{n+m}) \right)
\end{array} $$ 
such that
\begin{enumerate}
\item[(P1)] the waves generated by $\left( \left( \rho_{i,0} , v_{i,0} \right) , \left( \hat{\rho}_{i} , \hat{v}_{i} \right) \right)$ 
have negative speeds for every $i \in \{ 1, \ldots,n \}$ (incoming roads).
This means in particular that $\left( \hat{\rho}_{i} , \hat{v}_{i} \right)$  belongs to the curve of the first family 
passing through $\left( \rho_{i,0} , v_{i,0} \right)$.
\item[(P2)] the waves generated by $\left( \left( \hat{\rho}_{i} , \hat{v}_{i} \right) , \left( \rho_{i,0} , v_{i,0} \right) \right)$ 
have positive speeds for every $i \in \{ n+1, \ldots, n+m \}$ (outgoing roads).
\end{enumerate}
\end{definition}

One important feature of the Riemann solver is the \emph{consistency}, according to the following definition:
\begin{definition}[Consistency of a Riemann solver]
\label{def:consistency}
A Riemann solver $\mathcal{RS}$
%and the Riemann problem~\eqref{eq:ARZ}-\eqref{eq:ARZ-CI}.
is \emph{consistent} if
$$ \mathcal{RS} \left( \mathcal{RS} \left( (\rho_{1,0} , v_{1,0}) , \ldots , (\rho_{n+m,0} , v_{n+m,0}) \right) \right)  
= \mathcal{RS} \left( (\rho_{1,0} , v_{1,0}) , \ldots , (\rho_{n+m,0} , v_{n+m,0}) \right) . $$
%then $\mathcal{RS}$ is said to be \emph{consistent}.
\end{definition}

\subsubsection{Assumptions}

For sake of simplicity, we set $q_i := \hat{\rho}_i \hat{v}_i$ the flow at the junction on road $i$
for any $i \in \{ 1, \ldots, n+m \}$.

The Riemann solver will be based on the following assumptions:
\begin{enumerate}[label*=(A\arabic*)]
\item \underline{Mass conservation}: The sum of incoming fluxes equals the sum of outgoing fluxes
$$ \sum_{i=1}^{n} q_{i} = \sum_{j=n+1}^{n+m} q_{j}. $$
\item \underline{Fixed assignment coefficients}: 
there exist some fixed coefficients $\alpha_{ji}$ describing the preferences of the drivers.
These coefficients determine the percentage of the flux which passes from an incoming road $i$ to an outgoing one $j$.
We thus have
$$ q_{j} = \sum_{i=1}^{n} \alpha_{ji} q_{i} $$
with $0 \leq \alpha_{ji} \leq 1$ and $\sum_{j=n+1}^{n+m} \alpha_{ji} = 1$ for any $i \in \{ 1,\ldots, n \}$.
We set $A := \left( \alpha_{ji} \right)_{i,j}$ the matrix of all drivers' preferences.

We define the set of admissible states as 
{ \small
\begin{equation}
\label{eq:set-constraints}
\displayindent0pt
\displaywidth\textwidth
\Omega_{n\times m} := \left\{
(q_1, \ldots, q_n) \in \R^n \ \left|
\begin{array}{cc}
0 \leq q_i \leq \Delta_i,
&\forall i \in \{ 1, \ldots, n \} \\
\displaystyle 0 \leq q_j = \sum_{i=1}^{n} \alpha_{ji} q_i \leq \Sigma_j,
&\forall j \in \{ n+1,\ldots, n+m \}
\end{array}
\right.
\right\}
\end{equation}}
where $\Delta_i$ (resp.\ $\Sigma_i$) stands for the demand (resp.\ supply) on road $i$
defined by the function~\eqref{eq:demand} (resp.~\eqref{eq:supply}).
While the demands are explicitly computed as 
$$\Delta_i := \de_i (\rho_{i,0} , v_{i,0} + p_i (\rho_{i,0})), \quad \mbox{for any} \quad i = 1, \ldots, n,$$
the supplies require more  attention since their construction is implicit.
Indeed, having Remark~\ref{rem:lagrangian-attributes} in mind, the supplies on outgoing roads will depend on the mixture of Lagrangian attributes from incoming roads, i.e., they will depend on the solution $(q_1, \ldots, q_n)$ themselves.
The mechanism to get the correct supplies will be detailed in the different cases we study hereafter.

\item \underline{Multi-objective maximization of the fluxes}: 
the drivers behave such that each incoming flux is maximized
\begin{equation}
\label{eq:Pareto-optimization}
\max_{\Omega_{n\times m}} \left( q_1, \ldots, q_{n} \right).
\end{equation}
\end{enumerate}

\begin{enumerate}[label*=(A\arabic*),resume]
\item\label{A4} \underline{Priorities for the incoming roads}:
for any arbitrarily fixed $\mathbf{\P} = \left( \P_1, \ldots, \P_n \right)$, %$P \in \R^n$
with $\P_i \geq 0$, $i \in \{ 1, \ldots, n \}$ and $\displaystyle \sum_{i=1}^{n} \P_i = 1$, 
the solution of the Riemann solver lies on the Pareto front of the multi-objective maximization problem~\eqref{eq:Pareto-optimization} and
$\frac{q_i}{\sum_{j=1}^{n} q_j}$ is the closest to the given priority parameter $\P_i$
 for any $i = 1, \ldots, n$.
\end{enumerate}

It is noteworthy that the assumptions (A3) together with (A4) are totally new. 
In the literature, one usually assumes either that the sum of the incoming fluxes is maximized (the reader is referred to the book \cite{GaravelloBook2006}) 
or that the solution is given by the projection of a given priority vector on the feasible set (see for instance \cite{garavello2016riemann}).
In the case of general 1-to-$m$ diverges ($n = 1$ and $m \geq 1$),
our multi-objective optimization~\eqref{eq:Pareto-optimization} is equivalent to the maximization of the incoming flux, i.e., $\displaystyle \max_{\Omega_{1 \times m}} q_1$, 
since there is only one incoming road.

\subsection{Review of the literature}

While there is a huge quantity of papers dealing with the first order LWR model on networks
(see the book~\cite{GaravelloBook2006} and the references in~\cite{bressan2014flows}),
there are only a few papers in the literature that propose coupling conditions for second order traffic flow models.
Among those which deal more particularly with the ARZ model, we can highlight \cite{Garavello2006,HautBastin2007,HertyMoutariRascle2006, HertyRascle2006}.

%Garavello-Piccoli (2006)
In \cite{Garavello2006}, the authors define several rules for their Riemann solvers
and they prove that these Riemann solvers are well-posed for special junctions.
However, it is noteworthy that in their case 
the ``generalized momentum'' $y := \rho (v + p(\rho))$ is not conserved through the junction.

%Haut, Bastin (NHM 2007)
In the paper \cite{HautBastin2007}, both the mass and the flow momentum are conserved
but the values of $w_i$ for $i = n+1, ... n+m$ are computed as a convex combination 
of the $w_i$ of incoming roads
with respect to the \emph{demands} on the incoming roads.
This solver has been inspired by~\cite{Jin2003} and it has been reused in~\cite{ParzaniBuisson2012,SiebelMauserMoutariRascle2009}.
Such a Riemann solver is known not to be consistent in the sense of Definition~\ref{def:consistency}.
For instance, it suffices to consider a 2-to-1 merge with one incoming road in free-flow phase and the other one in congested situation.

%Herty, Rascle (2006)
In \cite{HertyRascle2006}, the authors propose the homogenization of an underlying microscopic model to solve the mixing problem of driver behaviors through the junction when there is more than one incoming road.
Differently from~\cite{Garavello2006}, they consider that the quantity $w = v + p(\rho)$ is conserved through the junction.
Moreover, the flux at the junction is not necessarily maximized.
%Herty, Moutari, Rascle (2006)
This has been modified in \cite{HertyMoutariRascle2006}.
However both in~\cite{HertyRascle2006} and~\cite{HertyMoutariRascle2006}, due to the homogenization process considered,
the authors introduce a pressure function $\tilde{p}_j(\rho)$ on outgoing roads
that may be different from the initial function $p_j(\rho)$. \\

Other second order traffic models on junctions have been studied in the literature such as the Phase Transition models in~\cite{colombo2010road,garavello2016riemann}.
The Riemann solver in~\cite{garavello2016riemann} is quite similar to the one proposed in~\cite{HertyRascle2006}.
 \\

Finally, in the engineering literature,
%Lebacque, Mammar, Haj-Salem (IFAC, 2008)
a junction model has been proposed in~\cite{lebacque2008intersection}
for models extracted from
the GSOM family \cite{lebacque2007generic, lebacque2008modelisation}
which encompasses a wide range of second order traffic flow models.
The authors assume that all the incoming fluxes mix through a \emph{buffer} at the junction
before exiting on the outgoing roads.
We believe that such a process is tractable for keeping things simple but it looses some information
about the mixture of the incoming $w_i$.

\subsection{Organization of the paper}

In Section~\ref{sec:modelling:coupling} 
we provide coupling conditions for junctions with a single incoming road, namely the simplest spatial discontinuity treated as a 1-to-1 junction and for a 1-to-2 diverge.
We then detail the case of a 2-to-1 merge in Section~\ref{sec:merge-coupling},
which requires more technicalities due to nonlinear constraints in the set of admissible states. Section~\ref{sec:numResults} is devoted to a numerical demonstration of two key features of the proposed merge model.
Finally, we give some concluding remarks in Section~\ref{sec:conclu}.

%--------------------------------------------------------------------------------------------------
\section{Coupling conditions for junctions with a single incoming road}
\label{sec:modelling:coupling}

In the following, we present the couplings between roads at the junction point for several types of junction
with $n = 1$ incoming road and $m \geq 1$ outgoing roads. 
The considered coupling conditions can be given in terms of mass flow $q=\rho v$ and ``momentum flow'' $q w$. 
The computation of the actual states at a junction is not necessary.

\subsection{1-to-1 junction}
We model a spatial discontinuity, like for instance a bottleneck, by a 1-to-1 junction, i.e., $n = m = 1$.
This case is interesting because it provides good insights for the construction of the supply function on outgoing roads.

\begin{figure}[htb]
\centering
	\includegraphics[width=0.4\textwidth]{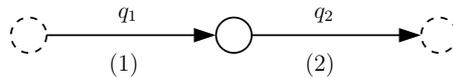}
    \label{fig:1to1}
    \caption{1-to-1 junction.}
\end{figure}

We use index 1 for the incoming road, index 2 for the outgoing road (see Figure~\ref{fig:1to1}), 
and consider the data $U_{i} = (\rho_{i}, \rho_{i} w_{i})$ at the adjacent boundaries of roads 1 and 2, respectively. 

In this very simple case, the assignment matrix is obviously $A = 1$ and the set of admissible states simply reads
$$ \Omega_{1\times 1} = \left\{ q_1 \in \R \ \left| \ \begin{array}{c}
0 \leq q_1 \leq \Delta_1 \\
0 \leq q_2 = q_1 \leq \Sigma_2
\end{array} \right. \right\} $$
where $\Delta_1 = \de_1\big( \rho_{1}, w_{1} \big)$ and $\Sigma_2$ 
denote the demand on the incoming road 
and the supply on the outgoing road, respectively.
The supply is not simply computed with respect to the initial conditions $(\rho_2,w_2)$ on the outgoing road.
Indeed, according to  Definition~\ref{def:RS}, we are looking for waves with positive speeds
on the outgoing road.
It embeds shocks, rarefaction waves and contact discontinuities.
We recall that the $w_i$ are advected with the speed of traffic.
Moreover a contact discontinuity propagates the change from $w_1$ to $w_2$ 
while the speed $v_2$ is a Riemann invariant and remains unchanged across the wave.
We thus compute a \emph{modified} density $\tilde\rho_2$ which is either given by the intersection of the curves $\{ w_2(U) = w_1 \}$ and $\{ v_2(U) = v_2(U_2) \}$ or by $\tilde{\rho}_2 = 0$.
It can be computed as follows:
\[
\tilde\rho_2 :=
p_2^{-1}\left(\max\left\{0,w_1-v_2\right\}\right).
\]
%which is either obtained by the intersection of the curves
%
%\begin{equation}
% \{ v_2(U) = v_2(U_{2}) \} \quad \text{and} \quad \{ w_2(U) = w_{1} \},
%\end{equation}
%
%or $\tilde \rho_2 = 0$, say we have the appearance of vacuum. 

Flow maximization at the junction over all admissible states leads to
\begin{equation}\label{eq:q1to1}
 q_1 = q_2 = \tilde q = \min\left\{ \de_1\big( \rho_{1}, w_{1} \big),\, \su_2\big( \tilde\rho_2, w_{1} \big) \right\},
\end{equation}
where the functions $\de_1$ and $\su_2$ are defined in~\eqref{eq:demand} and~\eqref{eq:supply}, respectively.
Obviously, the minimum in~\eqref{eq:q1to1} exists and is unique
and thus the Riemann solver is well-posed.

\begin{rem}
The solution~\eqref{eq:q1to1} for a 1-to-1 junction is totally equivalent to the Riemann solvers used for deriving numerical schemes for similar second order traffic flow models, namely the GSOM family in \cite{lebacque2007generic} and the Collapsed Generalized ARZ model in \cite{fan2017collapsed}.
\end{rem}

Note that for the considered pressure functions~\eqref{eq:pressure} the \emph{modified} density $\tilde \rho_2$
can be computed explicitly:
\begin{align}\label{rho:ex}
 \tilde \rho_2 = \rho_2^{\max} \left( \max \left\{ 0, \frac{\gamma_2}{\vref_2} \big(w_{1} - v_2(U_{2})\big) \right\} \right)^\frac{1}{\gamma_2} .
\end{align}
Then, $q_1 = q_2 = \tilde q$ determines the mass flow out of road 1/into road 2.
Further, with $\tilde w = w_{1}$ at the junction, the momentum flow is given by $\tilde q \tilde w$.
The densities $\left( \hat{\rho}_1, \hat{\rho}_2 \right)$ and speeds $\left( \hat{v}_1, \hat{v}_2 \right)$ are defined such that
\[
\hat{\rho}_i \left( w_1 - p_i \left( \hat{\rho}_i \right) \right) = \tilde{q}, \quad \mbox{and} \quad
\hat{v}_i = w_1 - p_i \left( \hat{\rho}_i \right), 
\quad \forall i \in \{ 1, 2\}.
\]

\subsection{Dispersing junction 1-to-2}

We assign index 1 for the incoming road, and indices 2 and 3 for the outgoing roads (see Figure~\ref{fig:1to2}).

\begin{figure}[htb]
\centering
	\includegraphics[width=0.4\textwidth]{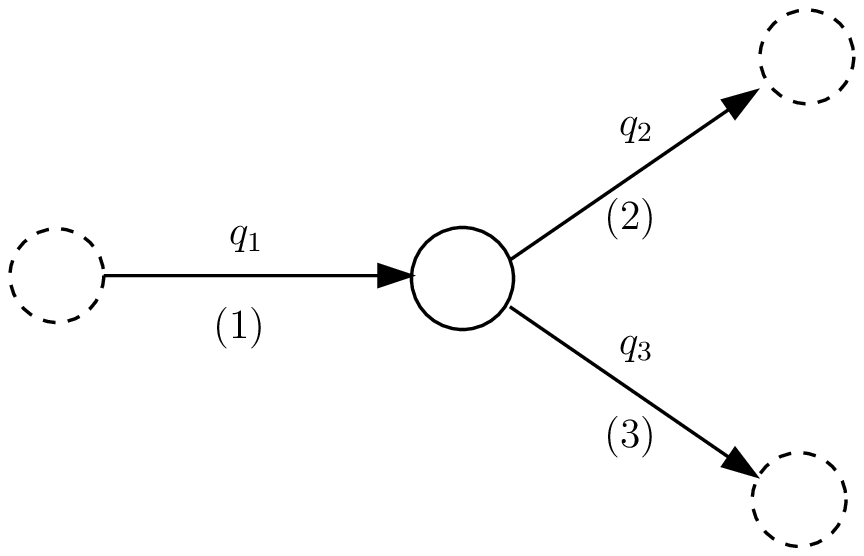}
    \label{fig:1to2}
    \caption{1-to-2 junction.}
\end{figure}

The assignment matrix is given by $A = \left( \alpha, 1-\alpha \right)$
where the distribution rate $\alpha \in \,]0,1[$ describes the proportion of the flow going from road 1 to road 2.
The set of admissible states is given by
$$ \Omega_{1\times 2} = \left\{ q_1 \in \R \ \left| \ \begin{array}{c}
0 \leq q_1 \leq \Delta_1 \\
0 \leq q_2 = \alpha q_1 \leq \Sigma_2 \\
0 \leq q_3 = (1-\alpha) q_1 \leq \Sigma_3
\end{array} \right. \right\} $$
where $\Delta_1 = \de_1\big( \rho_{1}, w_{1} \big)$ stands for the incoming demand and $\Sigma_j$, for $j=2, 3$, 
denote the supplies on both outgoing roads.

We exclude the extreme case of $\alpha=0$ (resp.\ $\alpha=1$) since in this case
we obtain $q_2 = 0$ and $q_1 = q_3$ (resp.\ $q_3 = 0$ and $q_1 = q_2$),
which reduces to a $1\times 1$ junction. In contrast with~\cite{HertyRascle2006}, we define
\begin{equation}
\label{eq:q1-to-2}
\begin{aligned}
 &q_1 = \min \left\{ 
 	\de_1\big( \rho_{1}, w_{1} \big),\, 
 	\dfrac{1}{\alpha} \su_2\big( \tilde\rho_2, w_{1} \big) ,
 	\dfrac{1}{1-\alpha} \su_3\big( \tilde\rho_3, w_{1} \big) \right\},\\
 &q_2 = \alpha q_1,\\
 &q_3 = (1-\alpha) q_1,
\end{aligned}
\end{equation}
where similar to above $\tilde\rho_j$ 
is obtained by
\begin{equation}
\label{eq:rho_i_diverge}
\tilde\rho_j :=
p_j^{-1}
\left(
\max\right\{
0, w_1-v_j
\left\}
\right),
\qquad j\in\{2,3\}.
\end{equation}
% intersection of the curves
% %
% \begin{equation}
% \label{eq:rho_i_diverge}
%  \{ v_i(U) = v_i(U_{i}) \} \quad \text{and} \quad \{ w_i(U) = w_{1} \},
% \end{equation}
% %
% or $\tilde\rho_i = 0$.
% %
$q_1$ determines the mass flow out of road 1, and $q_2$, $q_3$ are the mass flows into roads 2 and 3, respectively. 
They are uniquely defined thanks to~\eqref{eq:q1-to-2}.
Since we have \mbox{$\tilde w = w_{1}$} at the junction, 
the necessary momentum flow can be computed by multiplication of $q_1$, $q_2$ and $q_3$ with $\tilde w$ as above.

\begin{rem}[Generalization to a $1 \times m$ junction]
The above results can be easily generalized to a diverge with $n = 1$ incoming road and $m \geq 1$ outgoing roads.
It suffices to consider an assignment matrix 
$A = \left( \alpha_{2,1}, \ldots, \alpha_{j,1}, \ldots, \alpha_{m+1,1} \right)$
with $\alpha_{j,1} \in \,]0,1[$ for any $j \in \{ 2, \ldots, m+1 \}$ and
$\displaystyle \sum_{j = 2}^{m+1} \alpha_{j,1} = 1$.
The admissible set is thus given by
$$ \Omega_{1\times m} = \left\{ q_1 \in \R \ \left| \ \begin{array}{cl}
0 \leq q_1 \leq \Delta_1 & \\
0 \leq q_j = \alpha_{j,1} q_1 \leq \Sigma_j, & \forall j \in \{2,\ldots, m+1 \}
\end{array} \right. \right\}, $$
where $\Delta_1 = \de_1\big( \rho_{1}, w_{1} \big)$ and
$\Sigma_j = \su_j\big( \tilde\rho_j, w_{1} \big)$ for any $j \in \{2, \ldots, m+1\}$.

Then we obtain
\begin{equation}
\label{eq:q1-to-m}
\begin{aligned}
 &q_1 = \min \left\{ 
 	\de_1\big( \rho_{1}, w_{1} \big),\, 
 	\min_{j \in \{ 2, \ldots, m+1\} } \dfrac{1}{\alpha_{j,1}} \su_j\big( \tilde\rho_j, w_{1} \big) \right\},\\
 &q_j = \alpha_{j,1} q_1, \quad \forall j \in \{ 2,\ldots, m+1 \}.
\end{aligned}
\end{equation}
The modified densities $\tilde\rho_j$ are obtained as in~\eqref{eq:rho_i_diverge}.

% If there exists at least one $k \in \llbracket n+1, n+m \rrbracket$ such that $\alpha_{1,k} = 0$,
% then we should consider a $1 \times \tilde{m}$ diverge where 
% $\tilde{m} := m - \text{Card} \left( \left\{ j = n+1, \ldots, n+m \ : \ \alpha_{1,j} = 0 \right\} \right)$.

% If there exists a $l \in \llbracket n+1, n+m \rrbracket$ such that $\alpha_{1,l} = 1$,
% then the diverge simply behaves like a 1-to-1 junction since all the flow goes to road $(l)$.
\end{rem}

%--------------------------------------------------------------------------------------------------
\section{Coupling conditions for a 2-to-1 merging junction}
\label{sec:merge-coupling}

We use index~1 and~2 for the incoming roads, and index~3 for the outgoing road (see Figure~\ref{fig:2to1}).
This case is more involved. This is due to the fact that the conservation of the momentum flow produces 
some nonlinear constraints in the set of admissible fluxes.

\begin{figure}[htb]
\centering
	\includegraphics[width=0.45\textwidth]{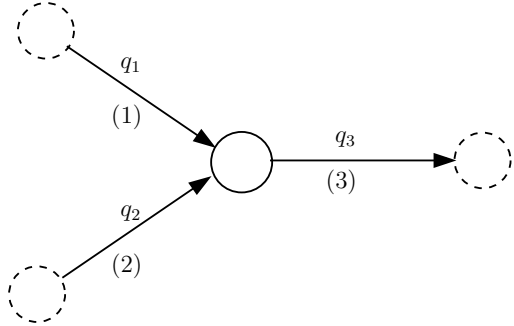}
    \label{fig:2to1}
    \caption{2-to-1 junction.}
\end{figure}

Indeed, in this case, the admissible set is given by
\begin{equation}
\label{eq:set_merge}
\Omega_{2\times 1} = \left\{ \left( q_1, q_2 \right) \in \R^2 \ \left| \ \begin{array}{c}
0 \leq q_1 \leq \Delta_1 \\
0 \leq q_2  \leq \Delta_2 \\
0 \leq q_3 = q_1 + q_2 \leq \Sigma_3
\end{array} \right. \right\},
\end{equation}
where $\Delta_1 = \de_1\big( \rho_{1}, w_{1} \big)$, $\Delta_2 = \de_2 \big( \rho_{2}, w_{2} \big)$
are the incoming demands and $\Sigma_3 = \su_3 \left( \tilde\rho_3, \tilde w \right)$ is the downstream supply.
The \emph{modified} density on the outgoing road
$\tilde\rho_3$ is given by 
\[
\tilde\rho_3 :=
p_3^{-1}
\left(
\max
\left\{
0, \tilde w - v_3
\right\}
\right),
\]
which, for the considered pressure functions~\eqref{eq:pressure},
%the \emph{modified} density $\tilde{\rho}_3$
reads
\begin{align}\label{eq:tilde_rho_3}
 \tilde \rho_3 = \rho_3^{\max} \left( \max \left\{ 0, \frac{\gamma_3}{\vref_3} \big( \tilde{w} - v_3 \big) \right\} \right)^\frac{1}{\gamma_3} .
\end{align}
% the intersection of the curves
% %
% \begin{equation}
%  \{ v_3(U) = v_3(U_{3}) \} \quad \text{and} \quad \{ w_3(U) = \tilde w \},
% \end{equation}
% %
% or it is given by $\tilde\rho_3 = 0$.
Above, $\tilde w$ is the convex combination of $w_1$ and $w_2$ given implicitly by
\begin{align} \label{eq:tilde_w}
 \tilde w = \frac{q_1}{q_1 + q_2} w_1 + \frac{q_2}{q_1 + q_2} w_2 ,
\end{align}
which depends on the final ``mixture rate'' ${q_1}/{q_2}$.
To overcome this difficulty, we propose to build a two-steps iterative process
to define a unique solution to the Riemann problem on a 2-to-1 merge,
based on our assumption~\ref{A4}.
Unlike other coupling conditions~\cite{HautBastin2007,%HertyMoutariRascle2006,
Jin2003,ParzaniBuisson2012,SiebelMauserMoutariRascle2009}, 
we assume that there exists a priority vector $\mathbf{\P} = (\P_1 , \P_2) = (\P, 1-\P)$ with $\P \in \, ]0,1[$ 
that is independent of the demand of the incoming roads.
It is noteworthy that a similar iterative process has been previously presented 
in~\cite{marigo2008fluid, monache2016priority} for the first order LWR model
and in~\cite{garavello2016riemann} for a Phase Transition model.

\subsection{Setting of the supply function}

Assume that there exists a parameter $\p \in [0,1]$ 
(that could be different to the priority parameter $\P$ introduced above) 
such that
$$ \begin{aligned}
&q_1 = \p q_3, \\
&q_2 = (1-\p) q_3,
\end{aligned} $$
where we recall that $q_3 = q_1 + q_2$.
We set $\Delta w = w_1 - w_2$ and we observe that
$$ \tilde{w} = \dfrac{q_1 w_1 + q_2 w_2}{q_1 + q_2} =
w_2 + \p \Delta w =: w(\p). $$

With simple algebra, 
putting~\eqref{eq:supply},~\eqref{eq:sonic_point},~\eqref{eq:tilde_rho_3} and~\eqref{eq:tilde_w} together,
the supply function $\Sigma_3 = \su_3 \left( \tilde{\rho}_3, \tilde{w} \right)$
can be written in a general form
$$ \Sigma_3 (q_1,q_2) = K \left( \dfrac{q_1 w_1 + q_2 w_2}{q_1 + q_2} + \delta \right)^{\gamma}, $$
or equivalently
$$ \tilde{\Sigma}_3 (\p) := K \left( w_2 + \p \Delta w + \delta \right)^{\gamma}, $$
with
$$ \begin{cases}
K = \left( \dfrac{\gamma_3}{\gamma_3+1} \right)^{\frac{\gamma_3+1}{\gamma_3}} 
\dfrac{\rho^{\max}_3}{\left( v^{ref}_3 \right)^{\frac{1}{\gamma_3}}}, \ \delta = 0, \ \gamma = \dfrac{\gamma_3+1}{\gamma_3} ,
\quad &\mbox{if} \quad w(\p) \leq \dfrac{\gamma_3+1}{\gamma_3} v_3, \\
K = v_3 \rho^{\max}_3 \left( \dfrac{\gamma_3}{v^{ref}_{3}} \right)^{\frac{1}{\gamma_3}}, \ \delta = -v_3, \ \gamma = \dfrac{1}{\gamma_3} ,
\quad &\mbox{if} \quad w(\p) > \dfrac{\gamma_3+1}{\gamma_3} v_3.
\end{cases} $$

Observe that
\begin{align*}
\Sigma_3 (q_1, q_2) = \tilde{\Sigma}_3 \left( \dfrac{q_1}{q_1 + q_2} \right) \quad \mbox{for any} \quad (q_1, q_2) \in \mathbb{R}^{+} \times \mathbb{R}^{+} \setminus (0,0).
\end{align*}

It is noteworthy that for $\Delta w \neq 0$ there exists  $\hat{\P}\in\mathbb{R}$ such that $w(\hat{\P}) = \dfrac{\gamma_3+1}{\gamma_3} v_3$. 
This particular priority value separates the two branches of the curve $\Sigma_3 (q_1, q_2) = q_1+q_2$ in the $(q_1,q_2)$ plane. 
It is explicitly given by
\begin{equation}\label{eq:hat_P}
\hat{\P} := \dfrac{1}{\Delta w} \left( \dfrac{\gamma_3+1}{\gamma_3} v_3 - w_2 \right).
\end{equation}

We observe that $\tilde{\Sigma}_3$ is continuously differentiable but not $C^2$ in $\hat{P}$.

\subsection{Convexity of the set of admissible states $\Omega_{2\times 1}$}

\begin{pro}[Convexity of the set of admissible states]
\label{pro:convexity}
The set of admissible states $\Omega_{2\times 1}$ defined in~\eqref{eq:set_merge} is non-empty and convex.
\end{pro}

\begin{proof}
First, we simply observe that $(0,0)$ belongs to $\Omega_{2\times 1}$.

Next, we consider two (distinct) points $$\Qt = \begin{pmatrix}\qat\\ \qbt\end{pmatrix} = \tilde{\Sigma}_3(\Pt) \begin{pmatrix}\Pt\\ 1-\Pt\end{pmatrix} \qquad \text{and} \qquad \Qtt = \begin{pmatrix}\qatt\\ \qbtt\end{pmatrix} = \tilde{\Sigma}_3(\Ptt) \begin{pmatrix}\Ptt\\ 1-\Ptt\end{pmatrix}$$ on the boundary induced by the supply function $\tilde{\Sigma}_3(\p)$, parameterized with respect to the flux ratio $\p = \frac{q_1}{q_1 + q_2}$ (see Figure~\ref{fig:convexity}). 
First neglecting restrictions due to demands $\Delta_1$ and $\Delta_2$, we will show that the segment between $\Qt$ and $\Qtt$ is feasible, i.e., $[\Qt,\Qtt] \subset \Omega_{2 \times 1}$. The convexity of the feasible set including demand restrictions ($q_1 \le \Delta_1$, $q_2 \le \Delta_2$) will follow from the fact that the intersection of convex sets is convex.

\begin{figure}[htb]
\centering
	\includegraphics[width=0.45\textwidth]{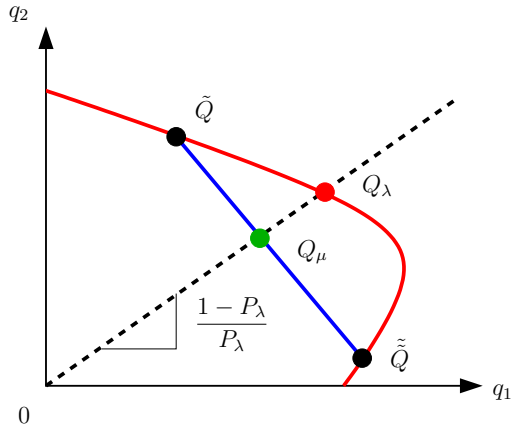}
    \label{fig:convexity}
    \caption{Sketch of the situation considered in the proof of Proposition~\ref{pro:convexity}.}
\end{figure}

Note that we consider the case
\begin{align*}
 \Delta w = w_1 - w_2 \ne 0.
\end{align*}
If $\Delta w = 0$, the boundary induced by the supply function is a straight line and the convexity of the feasible set is obvious.

Let
\[
 \Qmu = \Qt + \mu (\Qtt - \Qt) = \begin{pmatrix}q_1(\mu)\\ q_2(\mu)\end{pmatrix},
\qquad \mu\in [0,1]
\]
be the segment connecting $\Qt$ and $\Qtt$,
and set
\[
 \Pla = \Pt + \lambda \underbrace{(\Ptt - \Pt)}_{= \Delta P}, \qquad\lambda\in [0,1],
\]
such that $\Qla = \tilde{\Sigma}_3 \left( \Pla \right) \begin{pmatrix} \Pla \\ 1-\Pla \end{pmatrix}$ is on the same line through the origin as $\Qmu$.

We will show that the total flux corresponding to  $\Qmu$ is smaller than or equal to the total flux corresponding to the point $\Qla$, i.e.,
\begin{equation}\label{eq:inequality_convexity}
q_1 (\mu) + q_2 (\mu) \leq \tilde{\Sigma}_3 \left( \Pla \right),
\end{equation}
and thus the point $\Qmu$ is feasible (note that the feasible set is star-shaped with the origin as star center).

To compare the total fluxes corresponding to $\Qmu$ and $\Qla$, we first compute a representation of $\mu$ in terms of $\lambda$.
To simplify the notation, we introduce
\begin{align*}
 \St  = \tilde{\Sigma}_3(\Pt), \qquad
 \Stt = \tilde{\Sigma}_3(\Ptt), \qquad
 \Delta\Sigma = \Stt - \St.
\end{align*}

Since $\Qmu$ and $\Qla$ are on the same line through the origin, we have
\[
 \frac{q_1(\mu)}{q_1(\mu) + q_2(\mu)} = \Pla
\qquad \Leftrightarrow \qquad \frac{\Pt \St + \mu (\Ptt\Stt - \Pt\St)}{\St + \mu \Delta\Sigma} = \Pt + \lambda \Delta P.
\]

Multiplying both sides with $\St + \mu \Delta\Sigma$ and solving for $\mu$ yields (after a few algebraic transformations)
\begin{align}\label{eq:mu}
 \mu = \frac{\lambda\St}{(1-\lambda)\Stt + \lambda\St}.
\end{align}

Since the flux corresponding to $\Qmu$ is given by $\St + \mu \Delta\Sigma$, replacing $\mu$ from \eqref{eq:mu} yields
\[
 \St + \mu \Delta\Sigma = \frac{\St \Stt}{(1-\lambda)\Stt + \lambda\St}.
\]
To conclude the proof, we need to show~\eqref{eq:inequality_convexity}. 
This is equivalent to showing that
\[
 g(\lambda) := \tilde{\Sigma}_3(\Pla) \big( (1-\lambda)\Stt + \lambda\St \big) - \St \Stt
 \geq 0, \qquad \forall \lambda\in[0,1].
\]

Note that $g$ is continuously differentiable and that $g(0) = 0 = g(1)$. Thus, if there would exist a point $\lambda\in\, ]0,1[$ with $g(\lambda) < 0$, there would exist at least one local minimum of $g$ in $]0,1[$. We will show that such a point cannot exist.
%by considering all candidates $\las$ fulfilling $g'(\las)=0$. 
Note that $g$ is twice continuously differentiable in $]0,1[$ except for at most one point, i.e.\ $\hat{P}$ defined in~\eqref{eq:hat_P}, and the following proof also covers the case in which $P_{\las}=\hat{P}$. 
% Below we will show $g''(\las)<0$ in all stationary points $\las$. In the case that $\las$ equals the point where $g$ is not twice continuously differentiable, the argument is still valid for both one-sided limits of $g''$ (both smaller than $0$). Thus, all stationary points are local maxima and we get $g(\lambda) \ge 0$ for all $\lambda\in [0,1]$ as desired.

We have:
\begin{align*}
 \tilde{\Sigma}_3(\p) &= K \big( w(\p) + \delta \big)^{\gamma},\\
 \tilde{\Sigma}_3'(\p) &= K \gamma \big( w(\p) + \delta \big)^{\gamma-1} \Delta w,\\
 \tilde{\Sigma}_3''(\p) &= K \gamma (\gamma-1) \big( w(\p) + \delta \big)^{\gamma-2} (\Delta w)^2,
\end{align*}
where $w(\p) = w_2 + \p \Delta w$. 
%\textcolor{gray}{Note that the expression for $\tilde{\Sigma}_3''$ is also valid in the case $\gamma=1$, where it equals $0$.}

For the derivatives of $g$ we get
\begin{align*}
 g'(\lambda) &= \tilde{\Sigma}_3'(\Pla) \Delta P \big( (1-\lambda) \Stt + \lambda \St \big) - \tilde{\Sigma}_3(\Pla) \Delta\Sigma,\\
 g''(\lambda) &= \tilde{\Sigma}_3''(\Pla) (\Delta P)^2 \big( (1-\lambda) \Stt + \lambda \St \big) - 2 \tilde{\Sigma}_3'(\Pla) \Delta P \Delta\Sigma.
\end{align*}

Now, for any stationary point $\las$, we get from $g'(\las)=0$
\begin{align}\label{eq:ds}
 \Delta\Sigma = \frac{\tilde{\Sigma}_3'(\Plas) \Delta P \big( (1-\las) \Stt + \las \St \big)}{\tilde{\Sigma}_3(\Plas)}.
\end{align}
From \eqref{eq:ds}, we get for the (one-sided) second derivative(s) of $g$
\begin{align*}
 g''(\las) &= \tilde{\Sigma}_3''(\Plas) (\Delta P)^2 \big( (1-\las) \Stt + \las \St \big) - 2 \tilde{\Sigma}_3'(\Plas) \Delta P \Delta\Sigma\\
  &= K \gamma (\gamma-1) \big( w(\Plas) + \delta \big)^{\gamma-2} (\Delta w)^2 (\Delta P)^2 \big( (1-\las) \Stt + \las \St \big)\\
  & \quad - 2 K \gamma \big( w(\Plas) + \delta \big)^{\gamma-1} \Delta w \Delta P \frac{\tilde{\Sigma}_3'(\Plas) \Delta P \big( (1-\las) \Stt + \las \St \big)}{\tilde{\Sigma}_3(\Plas)}\\
  &= K \gamma (\gamma-1) \big( w(\Plas) + \delta \big)^{\gamma-2} (\Delta w)^2 (\Delta P)^2 \big( (1-\las) \Stt + \las \St \big)\\
  & \quad - 2 K^2 \gamma^2 \big( w(\Plas) + \delta \big)^{2\gamma-2} (\Delta w)^2 (\Delta P)^2
  \frac{\big( (1-\las) \Stt + \las \St \big)}{K \big( w(\Plas) + \delta \big)^{\gamma}}\\
  &=  \underbrace{K \big( w(\Plas) + \delta \big)^{\gamma-2}}_{>0} \underbrace{(\Delta w)^2}_{>0} \underbrace{(\Delta P)^2}_{>0} \underbrace{\big( (1-\las) \Stt + \las \St \big)}_{>0} \underbrace{\big( \gamma (\gamma-1) - 2 \gamma^2 \big)}_{=-(\gamma + \gamma^2) < 0} < 0.
\end{align*}
\end{proof}

Note that the convexity of the feasible set ensures the continuous dependence of the proposed Riemann solver based on~\ref{A4} with respect to the priority parameter $P$.

\subsection{Pareto front}
We are interested in the points $(q_1, q_2) \in \Omega_{2\times 1}$ such that both $q_1$ and $q_2$ are maximal.
The solution lies on the Pareto front, i.e., the set of points for which 
neither $q_1$ nor $q_2$ can be increased without decreasing the other one.
To this aim, we consider the local optima of 
$q_1 (\p) = \p \tilde{\Sigma}_3 (\p)$ and $q_2(\p) = (1-\p) \tilde{\Sigma}_3 (\p) = \tilde{\Sigma}_3(\p) - q_1(\p)$. 
We compute
\begin{align*}
q_1'(\p) &= \tilde{\Sigma}_3 (\p) + \p \tilde{\Sigma}_3' (\p) \\
&= K \left( w_2 + \p \Delta w + \delta \right)^{\gamma-1}
\left[ \Delta w (\gamma + 1) \p + w_2 + \delta \right]
\end{align*}
and
\begin{align*}
q_1''(\p) &= 2 \tilde{\Sigma}_3' (\p) + \p \tilde{\Sigma}_3'' (\p) \\
&= K \gamma \Delta w \left( w_2 + \p \Delta w + \delta \right)^{\gamma-2}
\left[ \Delta w (\gamma + 1) \p + 2 (w_2 + \delta) \right].
\end{align*}

Let
\begin{equation}
\label{eq:P*}
P^* := - \dfrac{w_2 + \delta}{(\gamma+1) \Delta w}
= \begin{cases}
- \dfrac{\gamma_3}{2 \gamma_3 +1} \dfrac{w_2}{\Delta w}
\quad &\mbox{if} \quad w_2 \leq \dfrac{2 \gamma_3+1}{\gamma_3} v_3, \\
\\
- \dfrac{\gamma_3}{\gamma_3 +1} \dfrac{w_2-v_3}{\Delta w}
\quad &\mbox{if} \quad w_2 > \dfrac{2 \gamma_3+1}{\gamma_3} v_3,
\end{cases}
\end{equation}
the point such that $q_1'(P^*) = 0$.

We observe that $\Delta w (\gamma + 1) P^* + w_2 + \delta = 0$
and we deduce that
$$ q_1'' (P^*) = K \gamma \Delta w \left( \dfrac{\gamma}{\gamma+1} \right)^{\gamma-2} (w_2 + \delta)^{\gamma-1}, $$
which has the same sign as $\Delta w$ 
(indeed, one can prove that $w_2 + \delta$ is always non-negative).
More precisely, if $\Delta w > 0$ then $P^*$ is a local minimum and
if $\Delta w < 0$ then $P^*$ is a local maximum for $q_1$.

Similarly, we can simply compute the local optima for $q_2 : \p \mapsto q_2(\p)$ 
by using the previous analysis of $\p \mapsto q_1(\p)$ and its derivatives, since 
\begin{align*}
q_2'(\p) &= \tilde{\Sigma}'_3(\p) - q_1'(\p) \\
&= - K \left( w_2 + \p \Delta w + \delta \right)^{\gamma-1}
\left[ \Delta w (\gamma + 1) \p + w_2 + \delta - \gamma \Delta w \right]
\end{align*}
and
\begin{align*}
q_2''(\p) &= \tilde{\Sigma}''_3(\p) - q_1''(\p) \\
&= -K \gamma \Delta w \left( w_2 + \p \Delta w + \delta \right)^{\gamma-2}
\left[ \Delta w (\gamma + 1) \p + 2 (w_2 + \delta) - (\gamma-1) \Delta w \right].
\end{align*}
We set $P^{**} := P^* + \dfrac{\gamma}{\gamma+1}$, say
\begin{equation}
\label{eq:P**}
P^{**} = \dfrac{\gamma}{\gamma+1} - \dfrac{w_2 + \delta}{(\gamma+1) \Delta w} 
= \begin{cases}
\dfrac{1}{2 \gamma_3+1} 
\left( 1 - \gamma_3 \dfrac{2 w_2 - w_1}{\Delta w} \right)
\quad &\mbox{if} \quad w_1 \leq \dfrac{2 \gamma_3+1}{\gamma_3} v_3, \\
\\
\dfrac{1}{\gamma_3+1} 
\left( 1 - \gamma_3 \dfrac{w_2-v_3}{\Delta w} \right)
\quad &\mbox{if} \quad w_1 > \dfrac{2 \gamma_3+1}{\gamma_3} v_3,
\end{cases}
\end{equation}
which satisfies $q_2'(P^{**}) = 0$.
% It is noteworthy that in the function $G_2 : w_1 \mapsto P^{**}$ is continuous and
% $$ G_2 \left( \dfrac{2 \gamma_3+1}{\gamma_3} v_3 \right) = 
% .$$

Using the fact that $\Delta w (\gamma + 1) P^{**} + w_2 + \delta - \gamma \Delta w = 0$, we compute 
$$q_2''(P^{**}) = -K \gamma \Delta w \left( \dfrac{\gamma}{\gamma+1} \right)^{\gamma-2}
\left( w_1 + \delta \right)^{\gamma-1}.$$
It is easy to observe that
if $\Delta w < 0$, then $P^{**}$ is a local minimum and 
if $\Delta w > 0$, then $P^{**}$ is a local maximum for $q_2$.

\begin{rem}
\label{rem:lim_P_star}
While this analysis is conducted for any value of $\Delta w$, we will only consider $P^*$ (respectively $P^{**}$) when $\Delta w < 0$ (resp.\ $\Delta w > 0$) since we want to maximize the fluxes $\left( q_1, q_2 \right)$ in accordance with assumption~\ref{A4}.

It is also interesting to observe that
$\displaystyle \lim_{\substack{\Delta w \to 0 \\ \Delta w < 0}} P^* = +\infty$ and
$\displaystyle \lim_{\substack{\Delta w \to 0 \\ \Delta w > 0}} P^{**} = -\infty$.
\end{rem}

We are now ready to properly set our Riemann solver.

\subsection{Definition of the Pareto-optimal priority-based Riemann solver}

The main idea of our construction principle is the following: 
starting from the flux ratio $\p = \P$ given by the priority parameter
and if the corresponding point on the boundary of the feasible set is not Pareto optimal itself,
we decrease (resp.\ increase) the flux ratio $\p$ if $\Delta w = w_1 - w_2 < 0$ (resp.\ if $\Delta w > 0$)
until we reach the closest point on the Pareto front.

Given initial Riemann data  on each branch of the junction,
%$\left( (\rho_{1,0}, v_{1,0}), (\rho_{2,0},v_{2,0}), (\rho_{3,0},v_{2,0}) \right)$, 
we define a vector $Q = (q_1, q_2)$ of incoming fluxes by a two-step procedure that is given as follows:
\begin{itemize}
\item \textbf{Step 1}: We first compute the theoretical flux on road 3 allowed if the priority parameter is enforced
\begin{align*}
F(\P) &= \min \left\{ \dfrac{\Delta_1}{\P} , \dfrac{\Delta_2}{1-\P} , \tilde{\Sigma}_3(\P) \right\}
\end{align*}
with
\begin{align*}
\Delta_i &= \de_i ( \rho_i, w_i ) \quad \mbox{for any} \quad i \in \{ 1, 2\}, \\
\tilde{\Sigma}_3 (\P) &= \su_3 ( \rho_\P , w_\P ), \\
w_\P &= w_2 + \P \Delta w.
\end{align*}
The value of $\rho_\P$ is given either by the intersection of the curves 
$\{ w = w_\P \}$ and $\{ v = v_3 \}$ 
or by $\rho_\P = 0$.
It is given by
\[
\rho_\P = p_3^{-1} \left( \max\{ 0, w_\P - v_3 \} \right)
\]
and it can be computed explicitly,
$$ \rho_\P = \rho_3^{\max} \left( \max \left\{ 0, \frac{\gamma_3}{\vref_3} \big( w_\P - v_3 \big) \right\} \right)^\frac{1}{\gamma_3} . $$

Further we compute the corresponding fluxes on both incoming roads
$$\tilde q_1 = \P F(\P) \qquad \text{and} \qquad \tilde q_2 = (1-\P) F(\P).$$

\item \textbf{Step 2}: We then distinguish the following different cases:

\begin{enumerate}

\item If
\begin{equation}\label{eq:conditionsEasyCase}
\begin{aligned}
 &\Delta w = 0, \qquad & \text{or}\\
 &\Delta w < 0 \quad \text{and} \quad \P \le P^*, \qquad & \text{or}\\
 &\Delta w > 0 \quad \text{and} \quad \P \ge P^{**},
 \end{aligned}
\end{equation}
we choose
\begin{equation}
\begin{aligned}\label{eq:resultEasyCase}
 q_1 &= \min \left\{ \Delta_1, \max \left\{
\tilde q_1 , \Sigma_3(q_1,q_2) - q_2 \right\} \right\}, \\
 q_2 &= \min \left\{ \Delta_2, \max \left\{
\tilde q_2 , \Sigma_3(q_1, q_2) - q_1 \right\} \right\}.
\end{aligned}
\end{equation}
Existence and uniqueness of this choice will be discussed in Section~\ref{sec:w1=w2} below as well as fulfillment of property~\ref{A4}.

\item\label{RS-hardCase} If
$$\Delta w < 0 \quad \text{and} \quad \P \ge P^*,$$
we compute
$$q_1^* = P^* \tilde\Sigma_3(P^*) \qquad \text{and} \qquad q_2^* = (1-P^*) \tilde\Sigma_3(P^*)$$
and apply
\begin{equation}\label{eq:hardCase-q2}
q_2 = \min \left\{ \Delta_2, \max \left\{
q_2^* , \Sigma_3(q_1, q_2) - q_1 \right\} \right\}.
\end{equation}
For the computation of $q_1$ we distinguish the following cases. Again, existence and uniqueness as well as fulfillment of property~\ref{A4} will be discussed in Section~\ref{sec:w_1<w_2} below.
\begin{enumerate}
 \item If
 $$F(\P) = \tilde\Sigma_3(\P) \qquad \text{and} \qquad q_2^* \le \Delta_2,$$
 we apply
 \begin{equation}\label{eq:hardCase-subcase1-q1}
  q_1 = \min\{ q_1^*, \Delta_1 \}.
 \end{equation}
 
 \item Otherwise
 \begin{equation}\label{eq:hardCase-subcase2-q1}
  q_1 = \min \left\{ \Delta_1, \max \left\{ \tilde q_1 , \Sigma_3(q_1,q_2) - q_2 \right\} \right\}.
 \end{equation}
\end{enumerate}

\item The case
$$\Delta w > 0 \quad \text{and} \quad \P \le P^{**}$$
can be treated analogously to case \ref{RS-hardCase}.
\end{enumerate}

From Remark~\ref{rem:lim_P_star}, we note that $\Delta w \rightarrow 0^-$ ensures $\P \le P^*$ and $\Delta w \rightarrow 0^+$ ensures $\P \ge P^{**}$, which is relevant for the continuity of the Riemann solver with respect to $\Delta w$. 

\end{itemize}

% {\color{blue}It is noteworthy that in~\eqref{eq:case_demand1} and~\eqref{eq:case_demand2}, we have a fixed point problem to solve. For instance, from~\eqref{eq:case_demand2}, we need to compute a solution to
% $$ q_2 = \Sigma_3(\Delta_1, q_2) - \Delta_1 = \tilde{\Sigma}_3 \left( \dfrac{\Delta_1}{\Delta_1 + q_2} \right) - \Delta_1. $$
% As the set of admissible states $\Omega$ is a convex set and since $\tilde{\Sigma}_3$ is a smooth function,
% we know that this fixed point problem has a unique solution.
% Hence our constructive method allows to ensure the existence and uniqueness of the solution.}

\subsection{Analysis for case \eqref{eq:conditionsEasyCase}} %$w_1 = w_2$}
\label{sec:w1=w2}

%We first begin with the case $w_1 = w_2$ or equivalently $\Delta w = 0$. 
%Here, $w(p)$ and $\tilde{\Sigma}_3(p)$ are constant. 
%First, we compute
Starting from
\begin{equation*}
 F(\P) =  \min \left\{ \frac{\Delta_1}{\P}, \, \frac{\Delta_2}{1-\P}, \, \tilde{\Sigma}_3(\P)\right\},
\end{equation*}
the following cases might occur:
\begin{enumerate}[label*=(E\arabic*)]
%%%%%%%%%%
% CASE 1 %
%%%%%%%%%%
 \item\label{case-E1} $F(\P) = \tilde{\Sigma}_3(\P)$ (Figure~\ref{fig:case1_1}): In this case, the point
 $$q_1 = \tilde q_1 = \P \tilde{\Sigma}_3(\P), \qquad q_2 = \tilde q_2 = (1-\P) \tilde{\Sigma}_3(\P)$$
 is Pareto optimal and fulfills the ``desired'' flux ratio ($\frac{q_1}{q_1+q_2} = \P$), \textbf{}thus~\ref{A4} is fulfilled.
 Obviously, the given point is a solution of~\eqref{eq:resultEasyCase}. Uniqueness of this solution follows from the strict monotonicity of the level set $q_2 = \Sigma_3(q_1,q_2) - q_1$ within $q_1, q_2 \ge 0$ in the considered case~\eqref{eq:conditionsEasyCase}.
 
 \begin{figure}[hbt]
 \centering
 \includegraphics{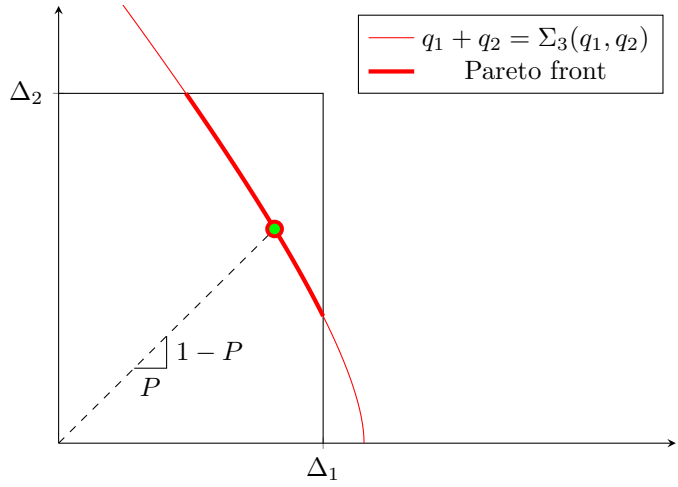}
 \caption{Illustration of the case \ref{case-E1}.}
 \label{fig:case1_1}
 \end{figure}
 
%%%%%%%%%%
% CASE 2 %
%%%%%%%%%%
 \item\label{case-E2} $F(\P) = \frac{\Delta_1}{\P}$ (Figure~\ref{fig:case1_2}): The desired solution according to~\ref{A4} is given by $q_1 = \Delta_1$ and the maximization of $q_2$ subject to
\begin{equation}\label{eq:maxBoundsq2}
 \tilde q_2 \le q_2 \le \min\left\{ \Sigma_3(\Delta_1, q_2) - \Delta_1, \Delta_2 \right\},
\end{equation}
%
%$$\quad q_2 = \min\{ \Delta_2, \, \tilde{\Sigma}_3(P) - \Delta_1 \},$$
which is Pareto optimal and where the flux ratio $\p = \frac{q_1}{q_1+q_2}$ is as close as possible to~$\P$.

Since $\tilde q_1 = \Delta_1$ here, $q_1 = \Delta_1$ actually is the unique solution of~\eqref{eq:resultEasyCase} for $q_1$ (independent of the value of $q_2$). To show that also the second equation of~\eqref{eq:resultEasyCase} has a unique solution for $q_2$ (with $q_1 = \Delta_1$), fulfilling the maximization within the bounds~\eqref{eq:maxBoundsq2}, we consider the function
\begin{equation}
 G_2(q_2) = \min \left\{ \Delta_2, \max \left\{ \tilde q_2 , \Sigma_3(\Delta_1, q_2) - \Delta_1 \right\} \right\} - q_2.
\end{equation}
Since $G_2$ is continuous and satisfies $G_2(\Delta_2) \le 0$ and 
$$G_2(\tilde q_2) = \min \{ \underbrace{\Delta_2}_{\ge \tilde q_2}, \underbrace{ \max\{ \tilde q_2 , \Sigma_3(\Delta_1, \tilde q_2) - \Delta_1\}}_{\ge \tilde q_2} \} - \tilde q_2 \ge 0,$$
existence of a solution 
within the bounds~\eqref{eq:maxBoundsq2} is clear. 
Note that the solution can be computed by bisection. 
Uniqueness again follows from the strict monotonicity of the level set $q_2 = \Sigma_3(q_1,q_2) - q_1$ within $q_1, q_2 \ge 0$ in the considered case~\eqref{eq:conditionsEasyCase}: there is at most one intersection of the curves $q_1 = \Delta_1$ and the level set in the feasible domain, corresponding to the solution $q_2 = \Delta_2$ (no intersection) or a unique point on the level set with $q_1 = \Delta_1$. In both cases the found solution maximizes $q_2$ within the bounds of~\eqref{eq:maxBoundsq2}.

\begin{figure}[ht!]
 \centering
 \includegraphics{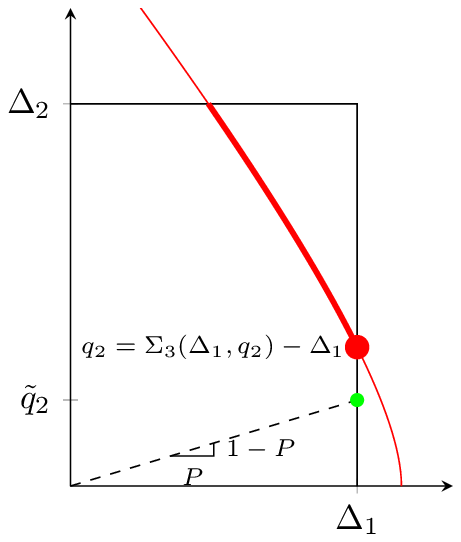}
 \includegraphics{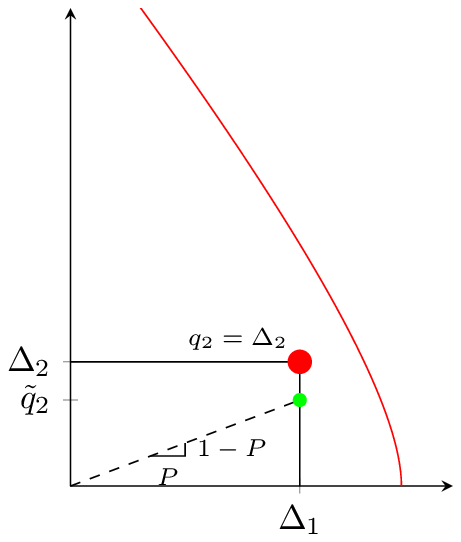}
 \caption{Illustration of the case \ref{case-E2}.}
 \label{fig:case1_2}
\end{figure}
  
%%%%%%%%%%
% CASE 3 %
%%%%%%%%%%
 \item\label{case-E3} $F(\P) = \frac{\Delta_2}{1-\P}$ (Figure~\ref{fig:case1_3}): Analogous to case~\ref{case-E2}, the desired solution according to~\ref{A4} is given by $q_2 = \Delta_2$ and the maximization of $q_1$ subject to
 \begin{equation}
  \tilde q_1 \le q_1 \le \min\{ \Sigma_3(q_1, \Delta_2) - \Delta_2, \Delta_1 \},
 \end{equation}
 which corresponds to the unique solution of~\eqref{eq:resultEasyCase} in this case.
 The proof of the existence and uniqueness of this solution is a straightforward adaptation of the case (E2).

\begin{figure}[ht!]
 \centering
 \includegraphics{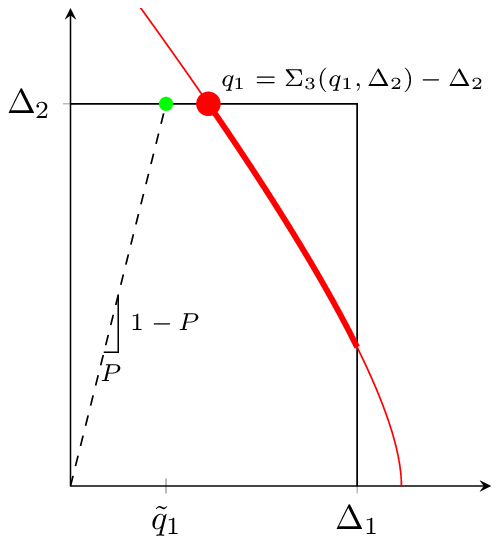}
 \includegraphics{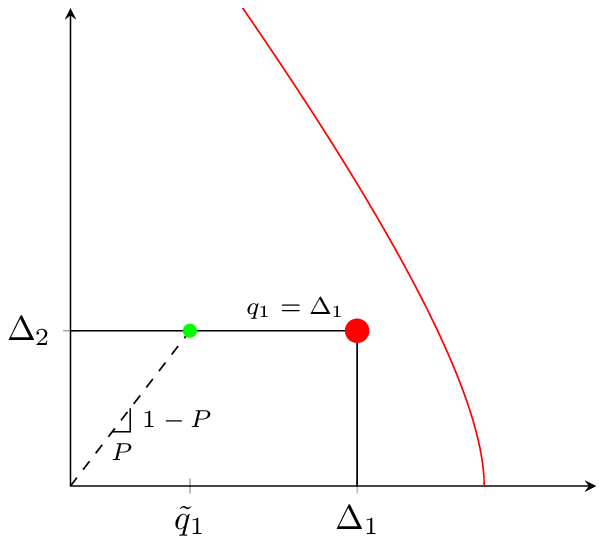}
 \caption{Illustration of the case \ref{case-E3}.}
 \label{fig:case1_3}
\end{figure}

 %we choose $$q_2 = \Delta_2, \quad q_1 = \min\{\Delta_1, \,  \tilde{\Sigma}_3(P) - \Delta_2 \},$$ which is again Pareto optimal and where the flux ratio $p = \frac{q_1}{q_1+q_2}$ is as close as possible to~$P$.

\end{enumerate}

% We can recover the expressions~\eqref{eq:case_demand1} and~\eqref{eq:case_demand2} given in the Riemann solver
% by simply observing that $\tilde{\Sigma}_3$ is a constant and thus
% $$ \tilde{\Sigma}_3(P) = \Sigma_3 (q_1, \Delta_2) = \Sigma_3 (\Delta_1, q_2) \quad \mbox{for any} \quad (q_1, q_2) \in \Omega. $$
% Moreover, since $0 \leq P \leq 1$, we have
% \begin{align*}
% \dfrac{1-P}{P} \Delta_1 \leq \tilde{\Sigma}_3(P) - \Delta_1
% \quad \mbox{if} \quad 
% \Delta_1 \leq P \tilde{\Sigma}_3(P),
% \end{align*}
% and
% \begin{align*}
% \dfrac{P}{1-P} \Delta_2 \leq \tilde{\Sigma}_3(P) - \Delta_2
% \quad \mbox{if} \quad 
% \Delta_2 \leq (1-P) \tilde{\Sigma}_3(P).
% \end{align*}

\subsection{Analysis for $\Delta w < 0$ and $\P \ge P^*$}
\label{sec:w_1<w_2}

First of all, notice that the case $\Delta w > 0$ with $\P < P^{**}$ can be treated analogously.

Next, we distinguish the following cases:
\begin{enumerate}[label*=(H\arabic*)]

 \item\label{case-H1} $F(\P) = \tilde{\Sigma}_3(\P)$ and $q_2^* \le \Delta_2$:
 \begin{enumerate}
 \item\label{case-H1-a} If additionally $q_1^* \le \Delta_1$, the point $(q_1^*,q_2^*)$ is feasible and fulfills the desired conditions~\ref{A4} (cf. Figure~\ref{fig:case2_2_1_1}). Further, it is also the unique solution of~\eqref{eq:hardCase-q2} and~\eqref{eq:hardCase-subcase1-q1}, since
 \begin{align*}
  \min\{ q_1^*, \Delta_1 \} = q_1^*,
 \end{align*}
 and $q_2 = q_2^* \le \Delta_2$ is the unique solution of $q_2 = \Sigma_3(q_1^*, q_2) - q_1^*$.
 
\begin{figure}[ht!]
 \centering
 \includegraphics{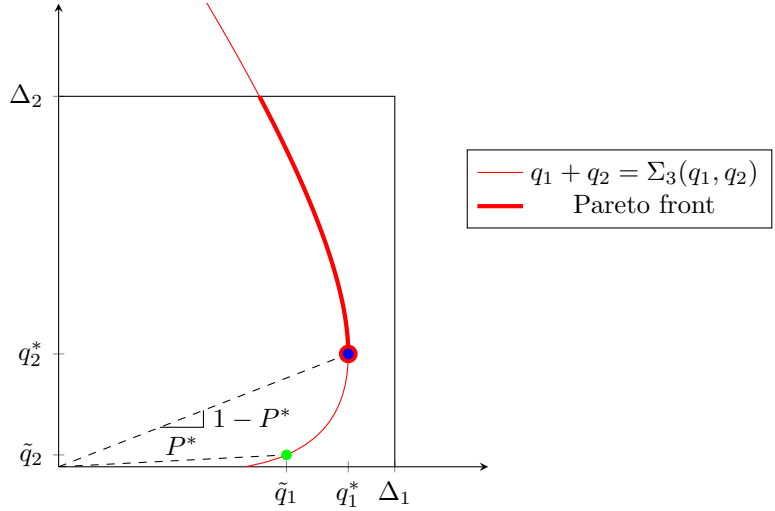}
 \caption{Illustration of the case \ref{case-H1-a}.}
 \label{fig:case2_2_1_1}
\end{figure} 
 
 \item\label{case-H1-b} If $q_1^* > \Delta_1$ (cf. Figure~\ref{fig:case2_2_1_2}), the desired solution according to \ref{A4} is $q_1 = \Delta_1$ and $q_2$ should be maximized subject to
 $$ q_2^* \le q_2 \le \min\{ \Delta_2, \Sigma_3(\Delta_1, q_2) - \Delta_1 \}.$$
 According to~\eqref{eq:hardCase-subcase1-q1}, $q_1 = \Delta_1$ is correctly chosen by our Riemann solver. Further, similar to the arguments in~\ref{case-E2}, the unique solution of~\eqref{eq:hardCase-q2} (with $q_1 = \Delta_1$) maximizes $q_2$ in the given bounds. Uniqueness here follows from the strict monotonicity of the level set for $q_2 \ge q_2^*$ and again bisection can be applied for the actual computation of the solution of~\eqref{eq:hardCase-q2}.
 
\begin{figure}[ht!]
 \centering
 \includegraphics{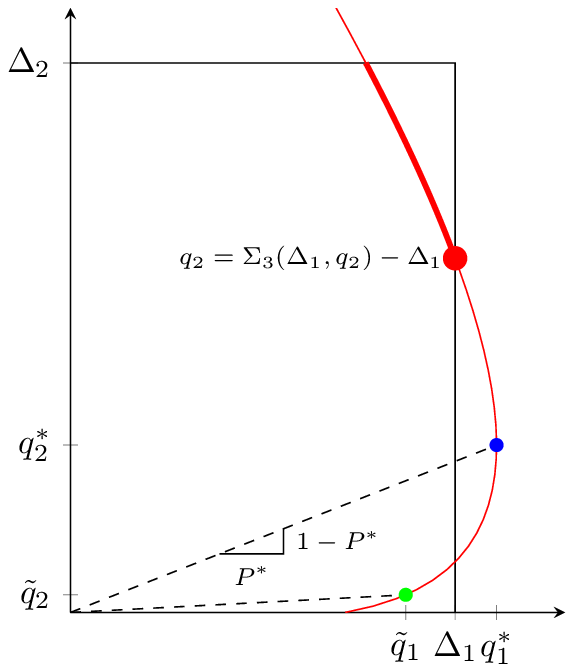}
 \includegraphics{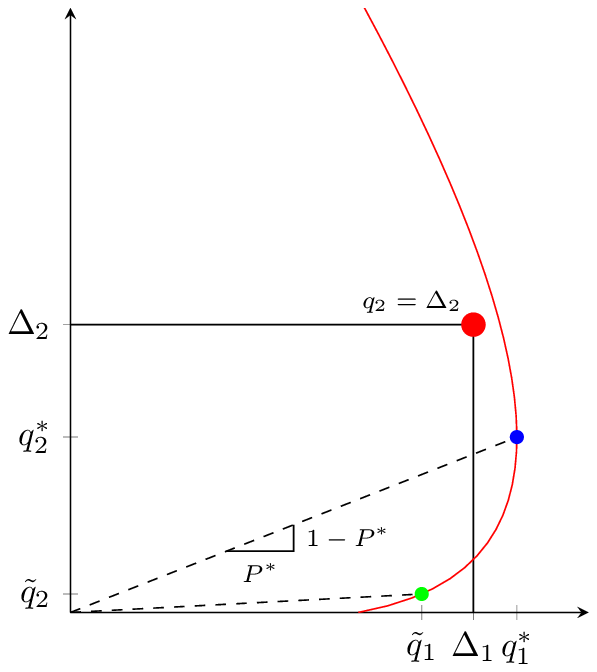}
 \caption{Illustration of the case \ref{case-H1-b}.}
 \label{fig:case2_2_1_2}
\end{figure} 

\end{enumerate}
 
 \item\label{case-H2} In the rest of the cases, the solution induced by~\ref{A4} should correspond to the (unique) solution of~\eqref{eq:hardCase-q2} and~\eqref{eq:hardCase-subcase2-q1}:
 
 \begin{enumerate}
  \item\label{case-H2-a} If $F(\P) = \tilde{\Sigma}_3(\P)$ and $q_2^* \ge \Delta_2$ (cf. Figure~\ref{fig:case2_2_1_3}): \ref{A4} induces $q_2 = \Delta_2$ and maximization of $q_1$ subject to
  $$ \tilde q_1 \le q_1 \le \min\{ \Delta_1, \Sigma_3(q_1,\Delta_2) - \Delta_2 \}. $$
  Since $q_2^* \ge \Delta_2$, \eqref{eq:hardCase-q2} directly yields
  $$ q_2 = \min\{ \Delta_2, \underbrace{\max \left\{
q_2^* , \Sigma_3(q_1, q_2) - q_1 \right\} }_{\ge q_2^*} \} = \Delta_2. $$
  With $q_2 = \Delta_2$ fixed, there is at most one solution of $q_1 = \Sigma_3(q_1,q_2) - q_2$ within the feasible set due to the strict monotonicity of the level set for $q_2 \le \Delta_2 \le q_2^*$. Further, $\Sigma_3(\tilde q_1,\Delta_2) - \Delta_2 \ge \tilde q_1$ so that either $q_1 = \Delta_1$ (no intersection with the level set) or the intersection point is the solution of~\eqref{eq:hardCase-subcase2-q1} - fulfilling the maximization of $q_1$ within the given bounds.
  
\begin{figure}[ht!]
 \centering
 \includegraphics{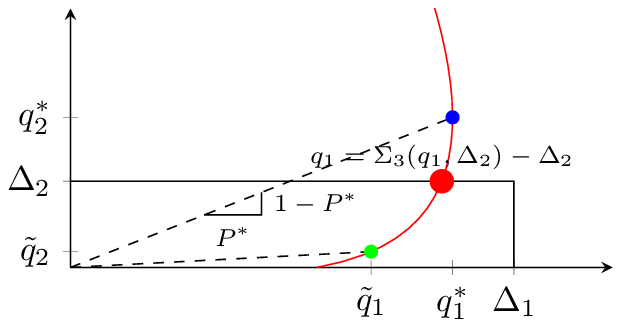}
 \includegraphics{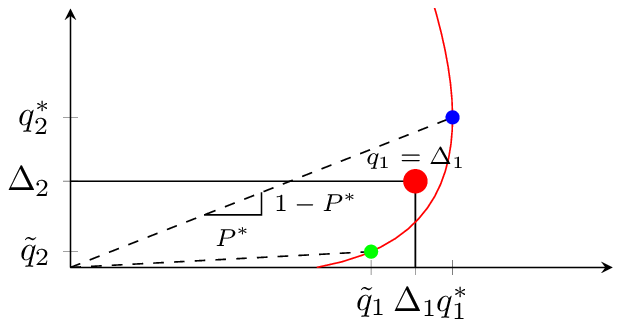}
 \caption{Illustration of the case \ref{case-H2-a}.}
 \label{fig:case2_2_1_3}
\end{figure} 
  
  \item If $F(\P) = \frac{\Delta_1}{\P}$: \ref{A4} induces $q_1 = \tilde q_1 = \Delta_1$ and maximization of $q_2$ subject to
  $$ \tilde q_2 \le q_2 \le \min\{ \Delta_2, \Sigma_3(\Delta_1,q_2) - \Delta_1 \}.$$
  First, \eqref{eq:hardCase-subcase2-q1} directly yields
  $$ q_1 = \min\{ \Delta_1, \underbrace{\max \left\{ \tilde q_1 , \Sigma_3(q_1,q_2) - q_2 \right\} }_{\ge \Delta_1} \} = \Delta_1$$
  as desired.
  
  Now, if $q_2^* \ge \Delta_2$, \eqref{eq:hardCase-q2} directly yields $q_2 = \Delta_2$, which maximizes $q_2$ in the given bounds. Otherwise, if $q_2^* \le \Delta_2$, \eqref{eq:hardCase-q2} only allows solutions $q_2 \in [q_2^*, \Delta_2]$. Again, we have at most one intersection of the level set $q_2 = \Sigma_3(q_1,q_2) - q_1$ with $q_1 = \Delta_1$ in the considered range, leading to a unique solution of~\eqref{eq:hardCase-q2}, which also maximizes $q_2$ in the given bounds. 
  
  \item If $F(\P) = \frac{\Delta_2}{1-\P}$: \ref{A4} induces $q_2 = \tilde q_2 = \Delta_2$ and maximization of $q_1$ subject to
  $$ \tilde q_1 \le q_1 \le \min\{ \Delta_1, \Sigma_3(q_1,\Delta_2) - \Delta_2 \}.$$
  Since we consider $\P \ge P^*$, we know $q_2^* \ge \tilde q_2 = \Delta_2$. Accordingly, \eqref{eq:hardCase-q2} directly yields $q_2 = \Delta_2$. Existence and uniqueness of a solution $q_1$ of~\eqref{eq:hardCase-subcase2-q1}, which additionally maximizes $q_1$ in the given bounds, follows again from the fact that there is at most one intersection of the level set $q_1 = \Sigma_3(q_1,q_2) - q_2$ with $q_2 = \Delta_2$ (and $\Sigma_3(\tilde q_1,\Delta_2) - \Delta_2 \ge \tilde q_1$).
  
 \end{enumerate}

 Note that the described solutions in the cases (H1) and (H2a) coincide if $q_2^* = \Delta_2$ (continuity of the proposed Riemann solver).

\end{enumerate}

\section{Numerical results}
\label{sec:numResults}

In this section we want to highlight two important features of the proposed merge model: First, it allows for the so-called capacity drop effect, i.e., an increase of ``desired'' fluxes on the incoming roads may lead to a decrease of the resulting accumulated outgoing flux. Secondly, the model allows for the exploitation of free capacities on the outgoing road since the flux ratio of the incoming roads is not strictly fixed by the given priority parameter.

We consider a simple merge situation as depicted in Figure~\ref{fig:2to1}. The two incoming roads with indices~1 and~2 have identical parameters $\rhomax_1 = \rhomax_2 = 180\,\ckm$, $\vref_1 = \vref_2 = 100\,\kmh$ and $\gamma_1 = \gamma_2 = 1.2$. The priority parameter is given by $P = 0.5$. The parameters of the outgoing road are $\rhomax_3 = 90\,\ckm$, $\vref_3 = 100\,\kmh$ and $\gamma_3 = 1.7$.

The initial states $(\rho_{i,0},v_{i,0})$ on all roads are chosen in such a way that
%an ``equilibrium state'', i.e.,
%
\begin{equation*}
 v_{i,0} = V_i(\rho_{i,0}) \quad \text{with} \quad V_i(\rho) = \vref_i \left( 1 - \frac{\rho}{\rhomax_i} \right).
\end{equation*}
The initial density on the first incoming road is given by $\rho_{1,0} = 30 \, \ckm$, resulting in a desired flux of $\rho_{1,0} v_{1,0} = 2500 \, \ch$.
On the second incoming road, we consider various densities $\rho_{2,0} \in [0,\frac{\rhomax_2}{2}]$ resulting in desired fluxes between $1000 \, \ch$ and $3500 \, \ch$.
On the outgoing road~3, we consider a low initial density of $\rho_{3,0} = 10 \, \ckm$.

The evolving situation at the merging point is computed by a numerical simulation based on a Godunov scheme as in~\cite{kolb2016capacity}, applying the proposed merge model at the common boundary of the three roads.
Table~\ref{tab:capacityDropMerge} reports the resulting fluxes at the merging point. The results are further illustrated in Figure~\ref{fig:capacityDrop}. Starting from a desired flux of about $1500 \, \ch$ on road 2, a further increase of the desired inflow results in a reduction of the accumulated outflow, i.e., the capacity drop effect takes place. Further note that the ratio of the resulting incoming fluxes is not fixed by the given priority parameter: Free capacity is exploited by road~1 in this example until the resulting flux on road~2 accounts for half of the resulting outgoing flux.

\begin{table}[hbt]
 \centering
 \caption{Capacity drop effect at a merge}
 \label{tab:capacityDropMerge}
 {\small
 \begin{tabular}{c|c||c|c||c||c|c}
  \multicolumn{2}{c||}{flux from road 1 in $[\ch]$} & \multicolumn{2}{c||}{flux from road 2 in $[\ch]$} & outflow & \multicolumn{2}{c}{flux ratio} \\
   \ \ desired \ \ & actual & \ \ desired \ \ & actual  & in $[\ch]$ & road 1 & road 2 \\
  \hline
  \hline
  2500.0 & 2500.0 & 1000.0 & 1000.0 & 3500.0 & 0.714 & 0.286 \\
  2500.0 & 2500.0 & 1400.0 & 1400.0 & 3900.0 & 0.641 & 0.359 \\
  2500.0 & 2413.1 & 1500.0 & 1500.0 & 3913.1 & 0.617 & 0.383 \\
  2500.0 & 2155.0 & 1750.0 & 1750.0 & 3905.0 & 0.552 & 0.448 \\
  2500.0 & 1945.3 & 2000.0 & 1945.3 & 3890.6 & 0.500 & 0.500 \\
  2500.0 & 1924.6 & 2500.0 & 1924.6 & 3849.3 & 0.500 & 0.500 \\
  2500.0 & 1903.9 & 3000.0 & 1903.9 & 3807.7 & 0.500 & 0.500 \\
  2500.0 & 1881.9 & 3500.0 & 1881.9 & 3763.8 & 0.500 & 0.500 
 \end{tabular}}
\end{table}

\begin{figure}[hbtp]
 \centering
 \includegraphics{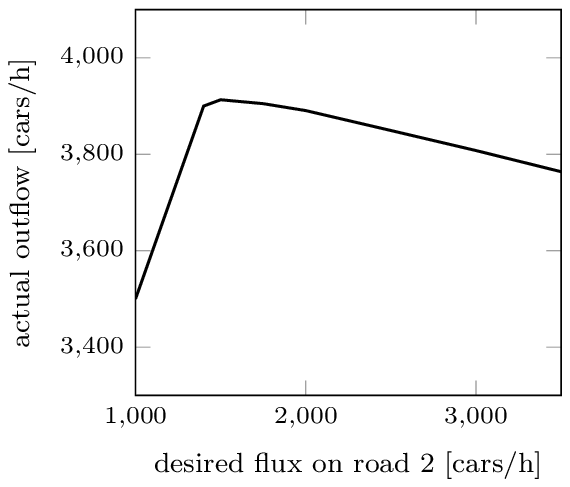}
 \quad
 \includegraphics{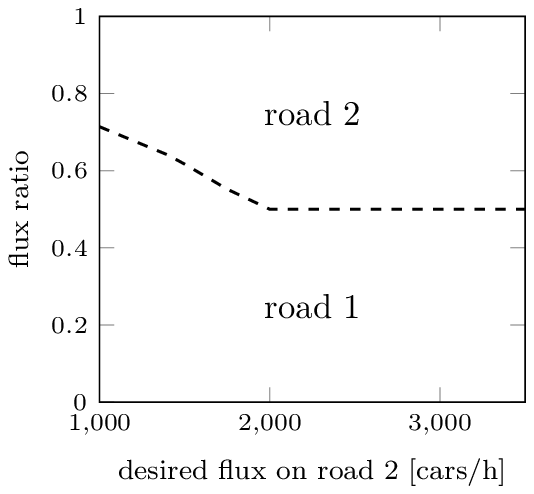}

  \caption{Actual outflow and flux ratio depending on the desired inflow on road 2.}
  \label{fig:capacityDrop}
\end{figure}

%--------------------------------------------------------------------------------------------------
\section{Conclusion}
\label{sec:conclu}

In this paper, we have presented coupling conditions for the Aw-Rascle-Zhang second order traffic flow model at junctions.
The solutions for general 1-to-m~($m \geq 1$) diverges and for a 2-to-1 merge are exhibited.
In the case of the merge, the solver is based on a two-steps construction making use of a priority parameter between  incoming roads.
This priority parameter is fixed and it is independent from the incoming demands.
The main contribution of our work is to present a multi-objective optimization problem of the incoming flows, leading to considering solutions lying on a Pareto front.
It is noteworthy that it is not straightforward to extend our well-posedness result for the 2-to-1 merge to $n$-to-1 merges or to general $n \times m$ junctions, since assumption \ref{A4} might not be sufficient to ensure the uniqueness of the solution. Further, convexity of the set of admissible states is not clear.

The proof of the consistency of our Riemann solver is out of the scope of this paper, since it involves long and cumbersome analysis. It will make the object of future studies.
Future research also includes further numerical computations for the presented Riemann solvers and applications to dynamic traffic management.

%\textcolor{red}{Another interesting research direction would be to couple the ARZ model with other models like the LWR model~\eqref{eq:LWR}, to account for different driving behaviors on on- or off-ramps for instance. Studying such couplings is left to future research.}

%--------------------------------------------------------------------------------------------------
\section*{Acknowledgments}

This work is partially supported by DFG grant GO 1920/4-1.

G. Costeseque is thankful to the Scientific Computing Research Group at the University of Mannheim for its hospitality during the preparation of this manuscript.

The authors thank the anonymous referees for indications to improve the paper.

%--------------------------------------------------------------------------------------------------
\bibliographystyle{siamplain}
\bibliography{references}

%--------------------------------------------------------------------------------------------------
\newpage
\tableofcontents

\end{document}